\documentclass[9pt]{article}

\makeatletter
\DeclareRobustCommand{\qed}{%
  \ifmmode 
  \else \leavevmode\unskip\penalty9999 \hbox{}\nobreak\hfill
  \fi
  \quad\hbox{\qedsymbol}}
\newcommand{\openbox}{\leavevmode
  \hbox to.77778em{%
  \hfil\vrule
  \vbox to.675em{\hrule width.6em\vfil\hrule}%
  \vrule\hfil}}
\newcommand{\qedsymbol}{\openbox}
\newenvironment{proof}[1][\proofname]{\par
  \normalfont
  \topsep6\p@\@plus6\p@ \trivlist
  \item[\hskip\labelsep\itshape
    #1.]\ignorespaces
}{%
  \qed\endtrivlist
}
\newcommand{\proofname}{Proof}
\makeatother

\usepackage[utf8]{inputenc}

\usepackage{bm}

\usepackage{color}
\usepackage{latexsym}
\usepackage{dsfont}
\usepackage{amssymb}
\usepackage{comment}
\usepackage{graphicx}
\usepackage{amsmath,amsfonts,amssymb,theorem,euscript,array,enumerate,amsfonts,mathrsfs}
\usepackage{hyperref}
\usepackage{appendix}
\usepackage[T1]{fontenc}
\usepackage{babel}
\numberwithin{equation}{section}
\usepackage{bbm}
\usepackage{subfigure}
\usepackage{color}
\usepackage{stmaryrd}

\usepackage[algo2e,ruled,vlined]{algorithm2e} 
 \usepackage{algorithm}
 \usepackage{algorithmicx}
\usepackage{algpseudocode}
\usepackage{ marvosym }
\usepackage{hyperref}
\usepackage{bm}
\usepackage{xcolor, soul}

\usepackage{mathtools}
\mathtoolsset{showonlyrefs}

\usepackage{comment}

\usepackage{tikz}
\usetikzlibrary{fit,matrix,chains,positioning,decorations.pathreplacing,arrows}

\usepackage{mathtools}

\usepackage{geometry}
\usepackage{algpseudocode}
\usepackage{algorithm}

\usepackage[normalem]{ulem}

\def \b1{\bf{1}}

\def \N{\mathbb{N}}
\def \R{\mathbb{R}}

\def \E{\mathbb{E}}
\def \F{\mathbb{F}}

\def \P{\mathbb{P}}

\def \S{\mathbb{S}}
\def \W{\mathbb{W}}
\def \H{\mathbb{H}}

\def \balpha{\boldsymbol{\alpha}}

\def \d{\mathrm{d}}

\def\esssup_#1{\underset{#1}{\mathrm{ess\,sup\, }}}

\def\argmin_#1{\underset{#1}{\mathrm{argmin\, }}}
\def\argmax_#1{\underset{#1}{\mathrm{argmax\, }}}

\def\dm#1{\frac{\delta}{\delta m}}

\def \Ac{{\cal A}}
\def \Bc{{\cal B}}
\def \Cc{{\cal C}}

\def \Fc{{\cal F}}

\def \Hc{{\cal H}}
\def \Ic{{\cal I}}
\def \Kc{{\cal K}}
\def \Lc{{\cal L}}
\def \Pc{{\cal P}}

\def \Sc{{\cal S}}

\def \Vc{{\cal V}}

\def\bX{{\bf X}}
\def\bY{{\bf Y}}
\def\bZ{{\bf Z}}

\def \d{\mathrm{d}}

\def\beqs{\begin{eqnarray*}}
\def\enqs{\end{eqnarray*}}
\def\beq{\begin{eqnarray}}
\def\enq{\end{eqnarray}}

\newcommand{\red}[1]{\textcolor{red}{#1}}
\newcommand{\bl}[1]{\textcolor{blue}{#1}}

\def\red#1{{\color{red}#1}}

\newtheorem{Theorem}{Theorem}[section] 
\newtheorem{Definition}[Theorem]{Definition} 
\newtheorem{Proposition}[Theorem]{Proposition}
\newtheorem{Assumption}[Theorem]{Assumption}
\newtheorem{Lemma}[Theorem]{Lemma}

\newtheorem{Remark}[Theorem]{Remark}

\title{Optimal  Control of  Heterogeneous  Mean-Field Stochastic Differential Equations with Common Noise and Applications}

\author{Filippo de Feo \footnote{Institut für Mathematik, Technische Universität Berlin, Berlin, Germany, Email: defeo@math.tu-berlin.de. This author acknowledges funding by the Deutsche Forschungsgemeinschaft (DFG, German Research Foundation) – CRC/TRR 388 "Rough Analysis, Stochastic Dynamics and Related Fields" – Project ID 516748464, by INdAM (Instituto Nazionale di Alta Matematica F. Severi) - GNAMPA (Gruppo Nazionale per l'Analisi Matematica, la Probabilità e le loro Applicazioni)} 
\and 
Samy Mekkaoui \footnote{Ecole Polytechnique, CMAP,  samy.mekkaoui at polytechnique.edu. This author is supported by the S-G Chair "Risques Financiers", and the "Deep Finance and Statistics" Qube-RT Chair.}}

\date{}

\begin{document}

\maketitle
\begin{abstract}
 We initiate the study of optimal control problems of heterogeneous mean-field stochastic differential equations with common noise. We formulate the problem within a linear–quadratic framework, a particularly important class in control theory, typically renowned for its analytical tractability and broad range of applications. We derive  a novel system of backward stochastic Riccati equations  on infinite-dimensional Hilbert spaces. As this system is not covered by standard theory, we establish existence and uniqueness of solutions. We explicitly characterize the optimal control in terms of the solution of this system. We apply these results to solve two problems arising in mathematical finance: optimal trading with heterogeneous market participants and systemic risk in networks of heterogeneous banks.
\end{abstract}




\section{Introduction}
\paragraph{Literature: standard mean-field SDEs, control,  and games.}
Understanding large populations and complex systems is a central challenge in mathematical modeling, with applications ranging from financial and economic systems to social networks, electrical grids, lightning networks, transport and mobility systems, communication infrastructures, biological and ecological populations, large-scale engineered networks, and many others. These systems typically evolve under multiple sources of uncertainty— idiosyncratic and systemic (common noise)—and the behavior of each agent often depends on the aggregate state of the population. This naturally leads to McKean–Vlasov dynamics and to the frameworks of mean-field games (MFG) and mean-field control (MFC), whose analysis is mathematically challenging due to the intrinsic coupling through the law of the state. In particular, when a common noise is present, the mean-field is no longer deterministic but becomes a random conditional distribution, from which substantial additional mathematical challenges typically arise. For a thorough introduction, we refer to \cite{lasry2007mean,bensoussan2013mean,carmona_probabilistic_2018,carmona_probabilistic_2018-1}. We also refer, e.g., to \cite{cosso2025mfos,talbi2023dynamic,talbi2023obstacle,talbi2024finite} for the closely related area of mean-field optimal stopping.

We recall that, among these models, a very popular class is the linear-quadratic class, renowned for its analytical tractability and broad spectrum of applications  \cite{yong2013lqmf,huang2015lqinfinite,sun2017mfslq,pham2016lqcmkv,basei2019weak},  leading to the study of suitable Riccati equations, typically  Riccati ordinary differential equations in the presence of idiosyncratic noise only and  Backward Stochastic Riccati equations (BSREs) in the context of common noise.


\paragraph{Literature: heterogeneous mean-field SDEs, control and games.}
 Motivated by the modeling of increasingly complex networks, the study of large-scale systems of agents with  heterogeneous interactions
has  grown rapidly, particularly through the graphon framework \cite{lovasz_large_2010}. Indeed, such framework was one of the main motivations to build a theory of heterogeneous mean-field systems  \cite{coppini2024nonlinear,bayraktar2023graphon,jabin2025mean,bayraktar2023propagation,bayraktar2024concentration,bayraktar2022stationarity,bayraktar_wu_2023graphon,bayraktar2025nonparametric,bayraktar2025graphon}, heterogeneous mean-field games
\cite{caines2020graphon,lacker2023label,xu2025lqgmfg,aurell2022stochastic,alvarez2025contracting,neuman2024stochastic}, heterogeneous mean-field control \cite{cao2025graphon,de2026mean,de2026linear,djete2025nonexchangeable,kharroubi2025stochastic,Mekkaoui2025analysis}, and heterogeneous mean-field teams \cite{fenghuhuang2023}, which are naturally infinite-dimensional formulations. In particular, to connect later with our work, we recall that  \cite{de2026linear} studied a linear-quadratic framework, which  led to the study of a suitable system of infinite-dimensional Riccati equations in the form of ODEs on the Hilbert spaces $L^2(I;\mathbb R^d)$ and $L^2(I\times I;\mathbb R^{d \times d})$, $I=[0,1]$, while
\cite{cao2025graphon,kharroubi2025stochastic} studied  the Pontryagin's maximum principle and introduced a collection of FBSDEs characterizing the optimal control problem. 
Finally, we refer to \cite{lauriere2025overview} for a very recent review on heterogeneous mean-field games.

While heterogeneous mean-field SDEs \cite{bayraktar2025graphon} and  games \cite{neuman2024stochastic,xu2025lqgmfg} with common noise have been studied, we observe that  mean-field control  with common noise remains largely open.
\paragraph{Literature: infinite-dimensional mean-field control and games.} On the other hand, extensions of standard mean-field theory have also been proposed for  mean-field games, mean-field control and optimal control of stochastic interacting particle systems  in infinite-dimensional spaces. Indeed, results have been obtained for  mean-field games on Hilbert spaces in \cite{federico_gozzi_ghilli_2024,federico_gozzi_swiech_2024,liu_firoozi_2024,firoozi_kratsios_yang_2025,fouque_zhang_2018,ricciardi_rosestolato_2024}, for mean-field control on Hilbert spaces in \cite{cosso_gozzi_kharroubi_pham_rosestolato,defeo_gozzi_swiech_wessels,djehiche_gozzi_zanco_zanella_2022,dumitrescu_oksendal_sulem_2018,gozzi_masiero_rosestolato_2024,spille_stannat_2025,tang_meng_wang_2019}. We  refer to \cite{shi_wang_yong_2013} for mean-field control of stochastic Volterra integral equations and to \cite{buckdahn_li_li_xing_2025,ren_tan_touzi_yang_2023} for path-dependent mean-field control and games. However, we remark that the infinite-dimensional nature of these formulations is very different from the ones modeling heterogeneity recalled in the previous paragraph, as here the goal is to model  an infinite-dimensional system under  homogeneity and anonymity.

\paragraph{Our theoretical contributions.}

In the present work, we begin the investigation of  optimal control problems of heterogeneous mean-field SDEs driven by common noise, which, as observed  above, has not yet been addressed in the literature. We formulate the problem within a linear–quadratic framework, a particularly important class in control theory, typically renowned for its analytical tractability and broad range of applications. 

In particular, we consider the following state equation
\begin{align}\label{eq : state equations LQ common noise_intro}
\begin{cases}
    &\d X^u_s =\left [ \beta_s^u + A_s^u X^u_s +  \int_I G_{s}^A(u,v)\E [ X_s^v | \tilde{\Fc}_s^0] \d v + B_s^u \alpha^u_s \right] \d s \\
    &\qquad\qquad +\left [\gamma_s^u + C_s^u  X^u_s +    \int_I G_{s}^C(u,v) \E [ X_s^v | \tilde{\Fc}_s^0] \d v + D^u_s \alpha^u_s \right ] \d W^u_s +\theta_s^u  \d \tilde{B}_s^0, \quad  s\in[t,T],\\
    & X_t^u=\xi^u, \qquad u \in I, 
\end{cases}
\end{align}
for an admissible collection $\boldsymbol{\xi} = (\xi^u)_{u \in I}$ of initial conditions, a collection of independent Wiener processes $W^u$ (acting as idiosyncratic noise), and an independent Wiener process $\tilde{B}^0$ (acting as a common noise).  The  goal is to minimize the following conditional cost functional over all admissible controls $\balpha = (\alpha^u)_{u \in I}$
\begin{align}\label{eq : cost_functional_intro}
     J(t,\boldsymbol{\xi},\balpha) &:=\;  \int_ I\E\Big[  \int_t^T \langle X^u_s,Q_s^u X^u_s \rangle + \langle \E [ X^u_s | \tilde{\Fc}_s^0 ],\int_I G_{s}^Q(u,v)\E [ X^v_s |\tilde{\Fc}_s^0 ]\d v\rangle + \langle \alpha^u_s + E_s^u,R_s^u (\alpha^u_s + E_s^u) \rangle  \notag \\
     & \qquad  \qquad  \qquad + \;  \langle X^u_T, H^u X^u_T \rangle + \langle \E [ X^u_T |\tilde{\Fc}_T^0 ], \int_I G^{H}(u,v) \E [ X^v_T |\tilde{\Fc}_T^0 ]d v \rangle  \big | \tilde{\Fc}_t^0\Big]\d u, 
\end{align}
where $\tilde{\F}^0 :=(\tilde{\Fc}_t^0)_{t \geq 0}$ is the natural filtration generated by $\tilde{B}^0$. Note that as the collection of i.i.d. Brownian motions $\big \lbrace W^u : u \in I \rbrace$ is  uncountable, we do not have the joint measurability of the map $I \times \Omega \ni (u,\omega) \mapsto X^u(\omega)$ \cite{sun2006llnk}. This may be recast using Fubini extensions \cite{bayraktar2025graphon}. However, in the present setting, it is enough to require that
the family of conditional distributions given the common noise $\tilde{B}^0$ is measurable.
We first prove well-posedness of the state equation \eqref{eq : state equations LQ common noise_intro} and the conditional cost functional \eqref{eq : cost_functional_intro} (see Theorem \ref{thm: existence_unicityX}). The proof relies on a suitable representation of the solution to \eqref{eq : state equations LQ common noise_intro} for which we show that the mapping $I \ni u \mapsto \P_{(X^u,W^u,\tilde{B}^0)} \in \Pc_2( \Cc^d_{[t,T]} \times \Cc_{[0,T]} \times \Cc_{[0,T]})$ is measurable. 

A standard method for treating LQ control problems in finite- or infinite‐dimensional  spaces is to express the cost functional as the sum of a  quadratic term in the control and a term that does not depend on the control, i.e. the so-called fundamental relation. This is typically achieved by deriving a suitable system of Riccati equations governing the evolution of the quadratic form. This strategy was indeed used in \cite{de2026linear} for heterogeneous mean-field control problems, leading to a system of Riccati ODEs on Hilbert spaces.  Here, we extend this method to the case of common noise. Indeed, we are able to prove a fundamental relation for the conditional (random) cost functional \eqref{eq : cost_functional_intro} (Proposition \ref{prop : fundamental_relation_common_noise}), leading to a novel system of triangular Riccati BSDEs on the Hilbert spaces $L^2(I  ; \R^{d })$ and $L^2(I \times I ; \R^{d \times d})$, i.e. \eqref{eq : standardRiccati K}-\eqref{eq : abstract Riccati barK}. We observe that Equation \eqref{eq : standardRiccati K} is an infinite system, indexed by $u \in I$ of decoupled BSRE with values in $\R^{d \times d}$ which can be solved independently for each label $u \in I$ using standard theory. However, 
we need to show that $K^u$ is $\tilde{\F}^0$-adapted and that the map $I \times \Omega \ni(u,\omega) \mapsto K^u(\omega) $ is jointly measurable, inspired by arguments  in \cite{kharroubi2025stochastic}. On the other hand, Equation \eqref{eq : abstract Riccati barK} is a new type of backward stochastic Riccati equation on the separable Hilbert space $L^2(I \times I ; \R^{d \times d})$ which, to the best of our knowledge, is new in the literature and cannot be treated  using the standard theory of BSDEs on Hilbert spaces (see e.g. \cite{confortola2007dissipative,fabbri2017stochastic}). The proof of existence and uniqueness for \eqref{eq : abstract Riccati barK} extends the methods of   \cite{de2026linear} (for a deterministic Riccati ODE) to a BSDE setting and is inspired by arguments  in \cite{guatteri2006backward}. Indeed, we prove  an a.s. a priori estimate on the  norm of the integral operator $T_{\bar K_t}$ associated to the solution (see Proposition \ref{prop : a_priori_estimate_barKt}) by using  the fundamental relation (Proposition \ref{prop : fundamental_relation_common_noise}). This fundamental relation allows us to construct  optimal feedback controls (see Theorem \ref{thm : optimal_control_form}).

\paragraph{Applications in mathematical finance.}
We apply the theory developed to two problems arising in mathematical finance in the context of heterogeneity and  common noise, i.e.  optimal trading and  systemic risk.

We introduce an optimal trading problem under heterogeneous market participants and common noise in a cooperative setting (see   \cite{almgren2000optimal} for a framework in standard optimal control and for a mean-field control setting  \cite{frikha2025actorcritic} under homogeneity and exchangeability of the traders). 
In particular, we consider a large bank with traders indexed by $u \in I$ who act under the guidance of a lead trader. We assume that the inventory of each trader $u \in I$ is given by the following SDE
\begin{align}
        \d X_t^u =   \alpha_t^u   \d t + \sigma_t^u \d W_t^u + \sigma_t^{0,u}   \d \tilde{B}_t^0 , \notag \quad
    X_0^u = \xi^u.
\end{align}
where $\sigma^u$ and $\sigma^{0,u}$ are $\tilde{\F}^0$-adapted processes denoting, respectively the idiosyncratic noise and a common noise on the inventory of the agent $u \in I$. The goal of the lead trader is to minimize  an aggregate conditional cost functional of the form 
\begin{align}\label{eq : conditional_cost_functional_opti_trading_intro}
     \int_I \E \Big[ \int_{t}^{T} (\alpha_s^u + P^u)^2  \d s + \lambda^u \big(  X_T^u - \int_I \tilde{G}^{\lambda}(u,v) \E [ X^v_T |\tilde{\Fc}_T^0 ] \d v  \big)^2  | \tilde{\Fc}_t^0 \Big] \d u,
\end{align}
where  $P^u$ and $\lambda^u $ are constants denoting, respectively the transaction price and risk-aversion parameter for agent $u$, and where $\tilde{G}^{\lambda} \in L^2(I \times I;\R)$ is a symmetric graphon modeling the interactions between traders $u$ and $v$. We remark that the conditional cost functional in \eqref{eq : conditional_cost_functional_opti_trading_intro} can be expressed in the form of \eqref{eq : cost_functional_intro} by Remark \ref{rmk : conditional_functional_cost}.  We solve the problem by applying the theory developed.

As a second application, we consider a systemic risk model. A model with a continuum of heterogeneous banks was recently proposed in \cite{de2026linear}, but the uncertainty there is driven solely by idiosyncratic noise. For systemic risk, however, it is essential to incorporate a common noise component, which captures aggregate shocks affecting the entire banking system. Indeed, this is classical in the standard mean-field framework of \cite{carmona2015systemic, pham2017dynamic} (we also refer, e.g., to \cite{djete2024meanfieldmutual, djete2024meanfieldgamemutual} for further developments on systemic risk in mean-field settings, and \cite{defeo2024sensitivity} for a sensitivity analysis). In this work, we fill this gap by extending the heterogeneous continuum-bank model of \cite{de2026linear} to include a common noise. In particular, we consider a model of interbank borrowing and lending of a continuum of heterogeneous banks, with the log-monetary reserve of each bank $u \in I$ given by
\begin{align}
        \d X_t^u =  \Big[ \kappa (X_t^u - \int_I \tilde{G}^{\kappa} (u,v) \E[ X_t^v | \tilde{\Fc}_t^0 ]   \d v ) + \alpha_t^u   \Big] \d t + \sigma_t^u \Big(  \sqrt{1- (p_t^u)^2}  \d W_t^u + \rho_t^u  \d \tilde{B}_t^0 \Big ), \notag \quad
    X_0^u = \xi^u,
\end{align}
where each bank controls its borrowing and lending rate with the central bank via the policy $\alpha^u$ and $\kappa \leq 0$ is the rate of mean-reversion in the interaction from borrowing and lending between the banks.
$\sigma^u$ is the $\tilde{\F}^0$-adapted volatility  process of the $u$-labeled bank reserve and  $\tilde{G}^{\kappa}\in L^2(I \times I;\R)$ is a symmetric graphon modeling the interactions between banks $u$ and $v$. Here, there is also a common noise $\tilde{B}^0$ which affects bank $u$ through the term $\sigma^u \rho^u$, where $\rho^u$ is an $\tilde{\F}^0$-adapted process valued in $[-1,1]$ acting as a correlation parameter.
The central bank aims to mitigate systemic risk by minimizing an aggregate conditional cost functional of the form
\begin{align}\label{eq : conditional_cost_functional_systemick_risk_intro}
      \int_I \E \Big[ \int_{t}^{T} \Big( \eta^u \big(X_s^u - \int_I \tilde{G}^{\eta}(u,v) \E [ X^v_s |\tilde{\Fc}_s^0 ] \d v \big)^2  + (\alpha_s^u)^2\Big) \d s+ r^u \big( X_T^u  - \int_I  \tilde{G}^r(u,v) \E [ X^v_T |\tilde{\Fc}_T^0 ] \d v  \big)^2| \tilde{\Fc}_t^0 \Big] \d u, 
\end{align}

where $\eta^u$ and $r^u$ are positive constants penalizing the deviation from the weighted average. Again, the conditional cost functional in \eqref{eq : conditional_cost_functional_systemick_risk_intro} can be expressed in the form of \eqref{eq : cost_functional_intro} by Remark \ref{rmk : conditional_functional_cost}. We solve the problem by applying the theory developed. 
\paragraph{Outline of the paper.} The remainder of the paper is organized as follows. In Section \ref{sec:notations}, we introduce the notations used throughout the paper. In Section \ref{sec:formulation}, we formulate the optimal control problem. In Section \ref{sec:solution}, we introduce the associated infinite-dimensional system of backward Riccati equations, establish existence and uniqueness results, and analyze their solvability. We also derive a fundamental relation by means of a square completion argument. Finally, in Section \ref{sec:applications}, we present two applications in mathematical finance in the context of optimal trading and systemic risk. 

\section{Notations}\label{sec:notations}
\begin{enumerate}
\item [$\bullet$] Given a Banach space $E$ endowed with the norm $\lVert.\rVert_{E}$ and its Borel $\sigma$-algebra $\Bc(E)$, $I$ a compact subset of $\R$ endowed with its Borel $\sigma$-algebra $\Bc(I)$ and  $\lambda(\d u ) = \d u $ the Lebesgue measure on $(I,\Bc(I))$, we introduce the following spaces:
\begin{enumerate}
    \item  $
    \begin{cases}
        L^2(I ; E) &:= \lbrace \phi :   I \ni u \mapsto \phi^u \text{ measurable and } \int_I |\lVert \phi^u \rVert|_{E}^2 \d u< + \infty \rbrace, \notag \\
    L^{\infty}_{}(I;E) &:= \lbrace \phi : I \ni u \mapsto \phi^u \text{ measurable and } \underset{u \in I}{\text{ ess sup }} \hspace{0.2 cm} \lVert \phi^u \rVert_{E}< + \infty \rbrace.
    \end{cases}$
    \item 

$
\begin{cases}
    L^2_{}(I \times I ; E) &:= \lbrace \phi  : I \times I\ni (u,v)\mapsto \phi(u,v) \text{ measurable and } \int_I \int_I  \lVert \phi(u,v) \rVert_{E}^2 \d u \d v   < + \infty \rbrace, \notag \\
    L^{\infty}_{}(I \times I;E) &:= \lbrace \phi : I \times I \ni (u,v) \mapsto \phi(u,v) \text{ measurable and } \underset{u,v\in I}{\text{ ess sup }} \hspace{0.2 cm} \lVert \phi(u,v) \rVert_{E}< + \infty \rbrace,
\end{cases}$
where we recall that $\underset{u \in I}{\text{ ess sup } } \lVert \phi^u \rVert_{E} = \text{ inf } \lbrace M\in \R : \lVert \phi^u \rVert_{E} \leq M \text{ $\d u $ a.e.} \rbrace$.
\item We denote by $\Pc(E)$ the space of probability measures on $(E,\Bc(E))$ and by $\Pc_{p}$ the subset of $\Pc(E)$ of the probability measures with finite $p$-th moments. Moreover, when $E$ is separable, we recall that checking the measurability of maps of the form $I \ni u \mapsto \mu^u \in \Pc_2(E)$ is equivalent to verifying that the maps of the form $
    I \ni u \mapsto \int_E \Phi(x) \mu^u(\d x) \in \R,$
are measurable for every choice of bounded continuous function $\Phi : E \to \R$.
\item We denote by $\Lc(E)$ the Banach space of bounded linear operators on $E$, with norm 
    $\lVert L \rVert_{\Lc(E)} = \underset{\lVert x \rVert_{E} = 1}{\sup } \lVert Lx \rVert_{E} =  \underset{\lVert x \rVert_{E} > 0}{\sup } \frac{\lVert Lx \rVert_{E}}{\lVert x \rVert_{E}}, $ $ L \in \Lc(E).$
    Recall that $\lVert \cdot \rVert_{\Lc(E)}$ is submultiplicative, i.e.
\begin{align}\label{eq : submultipliticativity_operator_norm}
    \lVert L \circ G \rVert_{\Lc(E)} \leq \lVert L \rVert_{\Lc(E)} \rVert G \rVert_{\Lc(E)}, \quad \forall  L,G \in \Lc(E),
\end{align}
where $\circ$ denotes the composition of operators.
\end{enumerate}

\item [$\bullet$] 
When $H$ is a separable Hilbert space endowed with scalar product $\langle \cdot , \cdot \rangle_{H}$, we say that a linear operator $L$ is self-adjoint if it satisfies 
$
    \langle Lf, g \rangle_{H} = \langle f, Lg \rangle_{H}, \forall (f,g) \in H.$
In this case, we also have 
\begin{align}\label{eq: operator_norm_def}
    \lVert L\rVert_{\Lc(H)} = \underset{ \lVert f \rVert_{H} > 0}{\sup}  \frac{ |\langle Lf, f \rangle_{H}|}{\lVert f \rVert_{H}^2} =\underset{ \lVert f \rVert_{H} = 1}{\sup}  | \langle Lf, f \rangle_{H}  |.
\end{align}

\item [$\bullet$] Let $d \in \N^*$ and $I:=[0,1]$ used to label the agents. We will denote by $\S^d$ (resp $\S^d_+,  \S^d_{> +}$) the set of $d$-dimensional symmetric (resp positive semi-definite, positive definite) matrices. For a given matrix $ M :=(M_{i,j})_{1 \leq i,j \leq d} \in \R^{d \times d}$, we will denote by $M^{\top}$ its transpose and by $\text{Tr}(M) = \sum_{i=1}^{d} M_{ii}$ its trace and for two vectors $x,y \in \R^d$, we denote by $\langle x,y \rangle = x^{\top} y = \sum_{i=1}^{d} x_i y_i$ their scalar product where $x_i$ denotes the $i$-th component of $x$. The Euclidean norm on $\R^{d \times m}$ will be denoted $| \cdot |$.

\item [$\bullet$] In the following, we will work mostly on the Hilbert spaces $L^2(I;\R^d)$ and $L^2(I \times I ; \R^{d \times d})$. 

Recall that these spaces are separable according to  \cite[Theorem 4.13]{brezis2011functional} with scalar products $\langle f, g \rangle_{L^2(I;\R^d)} = \int_I f(u)^{\top} g(u) \d u,$ $  f,g \in L^2(I;\R^d),$ 
    $\langle \bar{k},\bar{h} \rangle_{L^2(I \times I ; \R^{d \times d})}= \int_{I\times I} \text{Tr} \big(\bar{k}(u,v)^{\top} \bar{h}(u,v) \big) \d u \d v ,$ $ \bar{k}, \bar{h} \in L^2(I \times I ; \R^{d \times d}).$

We will now introduce the following notation which will be useful in the remainder of the paper.
\begin{enumerate}
    \item We will say that a kernel $K \in L^2(I \times I ; \R^{d \times d})$ is symmetric if the following holds
\begin{align}\label{eq : symmetric_kernels}
    K(v,u)^{\top} = K(u,v), \quad \text{ $\d u \otimes \d v$ a.e.} 
\end{align}
We denote by $L^2_{\text{sym}}(I \times I ;\R^{d \times d})$ the space of kernels $K \in L^2(I \times I; \R^{d \times d})$ satisfying \eqref{eq : symmetric_kernels}.
   \item  Given a kernel $K \in L^2(I \times I ; \R^{d \times d})$, we will denote  the  associated Hilbert-Schmidt linear integral operator $T_K$ as
\begin{align}\label{eq : operator_definition}
    L^2(I;\R^{d}) \ni f \mapsto T_K(f)(\cdot) = \int_I K(\cdot,v) f(v) \d v  \in L^2(I;\R^{d}).
\end{align}
We will say that $T_K$ is a symmetric non-negative  operator on $L^2(I;\R^d)$ if its associated kernel satisfies \eqref{eq : symmetric_kernels} and $\langle f, T_Kf \rangle_{L^2(I;\R^d)} \geq 0$ for every $f \in L^2(I;\R^d)$.
    \item Given a kernel $K \in L^2(I \times I ; \R^{d \times d})$, we define  $K^* \in L^2(I \times I ; \R^{d \times d})$ as
   \begin{align}\label{eq : tilde_kernel}
    K^*(u,v): = K(v,u)^{\top}.
    \end{align}
When $K \in L^2_{\text{sym}}(I \times I ; \R^{d \times d})$, we have $K^*= K$. We also note that
\begin{align}\label{eq : operator_adjoint}
    T_{K^*} = (T_K)^*, 
\end{align}
where $(T_K)^*$ denotes the adjoint operator of $T_K$ and that $\lVert K^{\star} \rVert_{L^2(I \times I ; \R^{d \times d})} = \lVert K \rVert_{L^2(I \times I ;\R^{d \times d})}$ and $\lVert T_{K^{\star}} \rVert_{\Lc(L^2(I;\R^d))} = \lVert (T_K)^{\star} \rVert_{\Lc(L^2(I;\R^d))} = \lVert T_{K} \rVert_{\Lc(L^2(I;\R^d))}$.
  \item Given two kernels $K,W \in L^2(I \times I ; \R^{d \times d})$, it follows directly that the  operator $T_{K} \circ T_{W}$ is associated with the kernel $\big(K \circ W \big)$ defined by
  \begin{align}\label{eq : kernel_product}
      \big(K \circ W \big)(u,v):= \int_I  K(u,w)W(w,v) \d w,
  \end{align}
  From the submultiplicativity of the operator norm in \eqref{eq : submultipliticativity_operator_norm},  we also have
  \begin{align}\label{eq : sub_multiplicativity_operator_norms}
    \lVert T_K \circ T_W \rVert_{\Lc(L^2(I;\R^d))} \leq \lVert T_K \rVert_{\Lc(L^2(I;\R^d))} \lVert T_W \rVert_{\Lc(L^2(I;\R^d))}.
\end{align}
  \item Given $L \in L^{\infty}(I ; \R^{d \times d})$, we define the multiplicative operator associated with $L$ as the linear operator on $L^2(I ; \R^{d})$ defined by 
  \begin{align}\label{eq : multiplicative_operator_definition}
    L^2(I;\R^{d}) \ni f \mapsto \Big( I \ni w \mapsto M_L(f)(w) = L^{w}f(w) \in \R^d \Big) \in L^2(I;\R^d).
\end{align}
Recall that $\lVert M_L \rVert_{\Lc(L^2(I;\R^d))} = \underset{w \in I}{\text{ess sup }} |L^w |$.

  Given $K,W \in L^2(I \times I ;\R^{d \times d})$ and $L \in L^{\infty}(I ;\R^{d \times d})$, the operator $T_K \circ M_L \circ T_W$ is associated with the kernel $\big(K \circ L \circ W \big)$ defined by
  \begin{align}\label{eq : kernel_form_multiplicativity}
    \big(K \circ L \circ W \big)(u,v) \ = \int_I K(u,w) L^w W(w,v) \d w  .
\end{align}
Again, from the submultiplicativity of the operator norm in \eqref{eq : submultipliticativity_operator_norm},  we also have
  \begin{align}\label{eq : sub_multiplicativity_operator_norms_with_multiplication}
    \lVert T_K \circ M_L \circ  T_W \rVert_{\Lc(L^2(I;\R^d))} \leq \lVert T_K \rVert_{\Lc(L^2(I;\R^d))} \lVert M_L \rVert_{\Lc(L^2(I;\R^d))} \lVert T_W \rVert_{\Lc(L^2(I;\R^d))}.
  \end{align}

\end{enumerate}

We recall below the inequality between the operator norm of a linear integral operator  and the $L^2$ norm.
\begin{align}\label{eq : operator_norm_inequality}
    \lVert T_K \rVert_{\Lc(L^2(I;  \R^{d }))}^2 \leq  \lVert K \rVert_{L^2_{}(I \times I ; \R^{d \times d})}^2, \quad \forall K \in L^2(I \times I; \R^{d \times d}).
\end{align}

\item [$\bullet$] Given a time horizon $T \geq 0$ and a normed vector space $(E, \lVert \cdot \rVert_E)$, we denote by $\Cc({[t,T]};E)$ the space of continuous functions $[t,T] \to E$ endowed with its Borel $\sigma$-algebra and its supremum norm $\lVert \omega \rVert_{\Cc({[t,T]};E)} = \underset{s \in [t,T]}{\sup} \lVert \omega_s \rVert_{E} $ for $\omega \in \Cc({[t,T]};E)$. When $E=\R^d$, we will only write $\Cc^d_{[t,T]}$ and when $d=1$, only $\Cc_{[t,T]}$. We denote by $\W_T$ the Wiener measure on $\Cc_{[0,T]}$. 
\item [$\bullet$] Given a Brownian motion $W:=(W_t)_{t \geq 0}$ valued in $\R^d$ defined on a complete filtered probability space $(\Omega,\Fc,\P)$, $\F =(\Fc_t)_{t \geq 0}$ its natural filtration and a separable Hilbert space $H$ with scalar product $\langle \cdot ,\cdot \rangle_{H}$, we introduce the following spaces
\begin{align}
\begin{cases}
    \S^2_{\F}([0,T];H) := \big \lbrace Y = (Y_t)_{0 \leq t \leq T} : Y \text{ is $\F$-progressively measurable and  $\E\big[ \underset{0 \leq t \leq T}{\sup} \lVert Y_t \rVert^2_{H} \big] < + \infty \rbrace$}, \\
    \H^2_{\F}([0,T];H) := \big \lbrace Z=(Z_t)_{0 \leq t \leq T} : Z \text{ is $\F$-progressively measurable and $\E \big[ \int_{0}^{T} \lVert Z_t \rVert^2_{H} \d t \big] < + \infty \rbrace.$ }
\end{cases}
\end{align}
Recall that $(\S^2_{\F}([0,T];H),\lVert \cdot \rVert_{\S^2_{F}([0,T];H)})$ and $(\H^2_{\F}([0,T];H),\lVert \cdot \Vert_{\H^2_{\F}([0,T];H)})$ are Banach spaces where $\lVert Y \rVert_{\S^2_{F}([0,T];H)} = \E\big[ \underset{0 \leq t \leq T}{\sup} \lVert Y_t \rVert^2_{H} \big] ^{\frac{1}{2}}$ and $\lVert Z \rVert_{\H^2_{\F}([0,T];H)} =\E \big[ \int_{0}^{T} \lVert Z_t \rVert^2_{H} \d t \big]^{\frac{1}{2}}$.

Since $T$ will be fixed throughout the paper, we will  denote $ \Sc^2_{\F}([0,T];H)$ and $\H^2_{\F}([0,T];H)$ respectively by $\Sc^2_{\F}(H)$ and $\H^2_{\F}(H)$ for simplicity.

We refer to \cite{fabbri2017stochastic} for a good introduction to stochastic analysis and control on separable Hilbert spaces.
\item [$\bullet$] Given a $\R^d$-valued Brownian motion $\tilde{B}^0:=(\tilde{B}^0_t)_{t \geq 0}$ with its natural filtration $\tilde{\F}^0=(\tilde{\Fc}_t^0)_{t \geq 0}$ acting as a common noise, we  introduce the following sets of processes

\begin{align}\label{eq : set_representation}
\begin{cases}
   \Sc^2_{\tilde{\F}^0}(I;E)  := \big \lbrace \bX:=(X^u)_{u \in I} : X_s^u = x(u,s,\tilde{B}^0_{. \wedge s})\forall s \in [t,T],  \d u \otimes \d \P \text{ a.e. for a Borel map $x  : I \times [t,T] \times \Cc^d_{[0,T]} \to E$} \\
\qquad \qquad \quad \qquad  \qquad \qquad \quad \text{ and $ \E \Big[ \int_I \underset{t \leq s \leq T}{\sup} \lVert |X_s^u \rVert|^2_{E}$} \d u \Big]  < + \infty \big \rbrace, \\
\Hc^2_{\tilde{\F}^0}(I;E) := \big \lbrace \bZ := (Z^u)_{u \in I } : Z_s^u = z(u,s,\tilde{B}^0_{. \wedge s}),\d u \otimes \d t \otimes \d \P \text{ a.e. for a Borel map $z : I \times [t,T] \times \Cc^d_{[0,T]} \to E $} \\ 
\qquad \qquad \quad \qquad  \qquad \qquad \quad \text{ and $ \E \Big[ \int_I \int_{0}^{T} \lVert |Z_s^u \rVert|_{E}^2$} \d t \d u   \Big] < + \infty \big \rbrace.
\end{cases}
\end{align}

Note that for any $\bX \in \Sc^2_{\tilde{\F}^0}(I;E)$ and $\bZ \in \Hc^2_{\tilde{\F}^0}(I;E)$, these processes are jointly measurable in $(u,t,\omega)$.
Moreover, we define the set 
\begin{align}
    \Sc^2_{\tilde{\Fc}_T^0}(I;E) &:= \big \lbrace \bX := (X^u)_{u \in I} : X^u = x(u,\tilde{B}^0_{. \wedge T}), \quad \d u \otimes \d \P \text{ a.e. } \text{for a Borel map $x: I \times \Cc^d_{[0,T]} \to E$} \\
    &\qquad \qquad \qquad \qquad \quad \text{ and }  \E \Big[ \int_I  \lVert X^u \rVert_{E}^2 \d u   \Big]  < + \infty \big \rbrace.
\end{align}

Similarly, we introduce the spaces $\Sc^2_{\tilde{\F}^0}(I \times I ; E), \Hc^2_{\tilde{\F}^0}(I \times I;E)$ and $\Sc^2_{\tilde{\Fc}_T^0}(I \times I;E)$. These sets will be used to ensure the well-posedness of the control problem \eqref{eq : state equations LQ common noise}-\eqref{eq:J_cost_functional}.

\item [$\bullet$] In the following, we write $C>0$ as a positive constant depending only on the data of the problem which can vary from line to line.

\end{enumerate}

\section{Problem formulation}\label{sec:formulation}
In this section, we introduce the problem. Let $(\Omega,\Fc,\P)$ be a complete probability space assumed to be rich enough to support an independent family $\lbrace W^u: u \in I \rbrace$ where for every $u \in I$, $W^u$ denotes a $\R$-valued standard Brownian motion.  We refer to \cite{sun2006llnk} for a construction of such a family.  We also consider, on the same probability space, respectively, a $\R$-valued Brownian motion $\tilde{B}^0$ and a uniform random variable $R$ on some compact $\mathcal{K} \subset \R$ assumed to be independent of each other  and of the family  $\lbrace W^u: u \in I \rbrace$. $\tilde{B}^0$ and $R$ will act, respectively, as a common noise and as a randomization for the initial condition. We  introduce $\F^u := (\Fc_t^u)_{0 \leq t \leq T}$ as the natural filtration generated by $W^u$ and $R$,  completed by the family of $\P$-null sets. Moreover, we denote by $\tilde{\F}^0 := (\tilde{\Fc}_t^0)_{0 \leq t \leq T}$ the natural filtration generated by $\tilde{B}^0$ completed by the family of $\P$-null sets. For the sake of simplicity, we consider dynamics driven by $\R$-valued Brownian motions as it simplifies the study of BSDEs on Hilbert spaces in Section \ref{sec:solution} but the results could be extended to the case of multi-dimensional Brownian motions. 
Finally, we introduce $\F^{u,0} = (\Fc_t^{u,0})_{0 \leq t \leq T}$ where $\Fc_t^{u,0} = \Fc_t^u \vee  \tilde{\Fc}_t^0$, the filtration generated by $(W^u,\tilde{B}^0,R)$ completed by the $\P$-null sets.

We assume that the controls are valued in $A = \R^m$ for $m \in \N^*$.
\paragraph{State equation.}
We consider a collection of  $\R^d$-valued state processes $\boldsymbol{X}= (X^u)_u$ , indexed by $u\in I=[0,1]$, and controlled by the collection 
$\balpha$ $=$ $(\alpha^u)_u$, driven by  
\begin{align}\label{eq : state equations LQ common noise}
\begin{cases}
    &\d X^u_s =\left [ \beta_s^u + A_s^u X^u_s +  \int_I G_{s}^A(u,v)\bar X^v_s \d v + B_s^u \alpha^u_s \right] \d s \\
    &\qquad\qquad +\left [\gamma_s^u + C_s^u  X^u_s +    \int_I G_{s}^C(u,v) \bar X^v_s\d v + D^u_s \alpha^u_s \right ] \d W^u_s+\theta_s^u  \d \tilde{B}_s^0,\quad s\in[t,T]\\
    & X_t^u=\xi^u, \qquad u \in I, 
\end{cases}
\end{align}
where we use the notation 
\begin{align}\label{eq : conditional_expectation_Xu}
    \bar X_s^u:=\mathbb E[X_s^u \mid \tilde{B}^0_{. \wedge s}],\quad s \in [t,T].
\end{align}
 Recalling the notations introduced above, we make the following assumptions on the model coefficients.

\begin{Assumption}\label{assumption : model coefficients}
\begin{enumerate} 
    \item [$\bullet$] $(\beta ,\gamma)\in \S^2_{\tilde{\F}^0}\big(L^2(I;\R^d) \big)$ and $\theta \in \S^2_{\tilde{\F}^0}(L^{2}(I;\R^{d })) \cap L^{\infty}(I \times \Omega; \Cc^d_{[0,T]})$.
    \item [$\bullet$] $(A,C) \in \S^2_{\tilde{\F}^0} \big(I;\R^{d \times d}\big) \cap L^{\infty}(I  \times \Omega;\Cc^{d \times d}_{[0,T]})  ,$
    \item [$\bullet$] $(G^A, G^C) \in \Sc^2_{\tilde{\F}^0}\big(L^2(I \times I ; \R^{d \times d} ) \big) $ and $(T_{G^A} ,T_{G^C}) \in L^{\infty} \big(\Omega ;\Cc([0,T];\Lc(L^2(I;\R^d))\big)$,
    \item [$\bullet$] $(B, D) \in \S^2_{\tilde{\F}^0}\big(L^{2}(I; \R^{d \times m }) \big) \cap L^{\infty}(I  \times \Omega; \Cc^{d \times m }_{[0,T]})$.
\end{enumerate}
\end{Assumption}

\noindent We now give the main structures which ensure the well-posedness of the controlled system \eqref{eq : state equations LQ common noise}.

\begin{Definition}[\textnormal{Admissible conditions}]\label{def : admissible_conditions}
\begin{itemize}
    \item [(1)] We say that $\boldsymbol{\xi}=(\xi^u)_u$ is an admissible condition if there exists a Borel measurable function $\xi : I \times [0,T] \times \Cc_{[0,T]} \times \Cc_{[0,T]} \times \mathcal{K}  \to \R^d$ such that we have for a.e. $u \in I$
    \begin{align}\label{eq : hypothesis initial conditions}
        \xi^u = \xi(u,t,W^u_{. \wedge t}, \tilde{B}^0_{. \wedge t},R),  \quad  \P \text{-a.s, }\text{and } \int_I \E \big[ | \xi^u|^2 \big] \d u < + \infty.
    \end{align}
    Such an initial condition is said to be admissible and we denote by $\Ic_t$ the space of admissible initial conditions.
    \item [(2)] We say that $\balpha =(\alpha^u)_u$ is an admissible control if there exists a Borel measurable function $\alpha : I  \times [0,T] \times \Cc_{[0,T]} \times \Cc_{[0,T]} \times \mathcal{K}  \to \R^m$ such that for a.e. $u \in I$
    \begin{align}\label{eq : hypothesis admissible control}
        \alpha_s^u = \alpha(u,s,W^u_{. \wedge s}, \tilde{B}^0_{. \wedge s},R), \quad \P- \text{a.s,} \quad \forall s \in [0,T],\text{ and }   \int_I \int_{0}^{T} \E \big[ |\alpha_s^u |^2  \big ] \d s \d u < + \infty.
    \end{align}
    Such a control is said to be admissible and we  denote by $\Ac$ the space of admissible controls.
\end{itemize}
\begin{Remark}\label{rmk : well_posedness}
\begin{enumerate}
    \item  The admissible conditions \eqref{eq : hypothesis initial conditions} and \eqref{eq : hypothesis admissible control} are chosen similarly to \cite{kharroubi2025stochastic} to ensure the well-posedness of  the control problem as the cost functional will be defined up to the joint probability law  of the state and control processes.
    \item   We point out that unlike the discussion in \cite{de2026linear}, the addition of the common noise in our current setting requires for every $s \in [t,T]$ the joint measurability of a version of the map $I \times \Omega \ni (u,\omega) \mapsto \E \big[X_s^u |\tilde{B}_0 \big] (\omega)$. However, noticing that $\E \big[X_s^u | \tilde{B}_0 \big](\omega) = \int_{\R^d} x d \P_{X_s^u|\tilde{B}^0(\omega)}(\d x)$ for a regular conditional version of the conditional law, the joint measurability can be obtained under the measurability of the joint law $I \ni u \mapsto \P_{(X_s^u, \tilde{B}^0)} \in \Pc_2(\R^d \times \Cc_{[0,T]})$ and the usual disintegration theorem on Polish spaces. Indeed, we construct the measure $\mu$ over the Polish space $I \times \R^d \times \Cc_{[0,T]}$ as follows 
    \begin{align}
        \mu(A \times B) = \int_{A} \mathbb{P}_{(X_s^u, \tilde{B}^0)}(B) \lambda(\d u ), \quad A \in \Bc(I), \quad B \in \Bc(\R^d \times \Cc([0,T]),
    \end{align}
    and where we use for now that the map $I \ni u \mapsto \mathbb{P}_{(X_s^u, \tilde{B}_0)} \in \Pc_2(\R^d \times \Cc_{[0,T]})$ is measurable for any $t \leq s \leq T$.  Hence, we now disintegrate the measure with respect to its coordinates $(u,y)$. Namely, denoting the map $\text{pr}_{13}(u,x,y) :=(u,y)$, we know by the disintegration theorem over Polish spaces  (see Theorem 1.1 in \cite{carmona_probabilistic_2018-1}) the existence of a  measurable kernel $k : I \times \Cc_{[0,T]} \ni (u,y) \mapsto \Pc(\R^d)$ such that $\mu(\d u , \d x , \d y) = k(u,y, \d x) \text{pr}_{13} \sharp \mu (\d u ,\d y)=k(u,y,\d x) \big( \lambda(\d u ) \otimes \W_T(\d y) \big)$, where we have used that  $\text{pr}_{13} \sharp \mu(\d u , \d y) = \lambda(\d u ) \otimes \W_T(\d y)$. Moreover, we  have $\mu(\d u ,\d x, \d y) = \lambda(\d u ) \mathbb{P}_{(X_s^u,\tilde{B}_0)}(\d x, \d y)=  \P_{X_s^u| \tilde{B}_0=y}(\d x) \big(\lambda (\d u) \otimes  \W_T(\d y) \big)$, where we have used the disintegration theorem once again. Thus, we can identify  $k(u,y, \d x)$ as a regular version of the conditional law of $\mathbb{P}_{X_s^u| \tilde{B}_0=y}(\d x)$ which is therefore jointly measurable. Hence, since $\tilde{B}^0$ is measurable, we have that
    \begin{align}
     \E[ X_s^u | \tilde{B}_0](\omega) =  \int_{\R^d} x \P_{X_s^u | \tilde{B_0}(\omega)}(\d x ) = \int_{\R^d} x  k(u,\tilde{B}_0(\omega) , \d x)),
    \end{align}
    which is jointly measurable in $(u,\omega)$.
\end{enumerate}

\end{Remark}

\noindent We now give a definition of a class of solutions for the problem formulation \eqref{eq : state equations LQ common noise}.
\begin{Definition}\label{def : solution space S_t}
   Given $t \in [0,T]$, we say that a family $\boldsymbol{X}:=(X^u)_{u \in I}$ of stochastic processes valued in $\R^d$ belongs to the space $\Sc_t$ if

   \begin{itemize}
       \item[(1)] The map $u \mapsto   \P_{(X^u,W^u,\tilde{B}^0,R)}$ is Borel measurable from $I$ to $\Pc_2( \Cc^d_{[t,T]} \times \Cc_{[0,T]} \times \Cc_{[0,T]} \times \mathcal{K} ) $.
       \item [(2)] Each process $X^u$ is continuous and $\F^{u,0}$-adapted.
       \item [(3)] The following norm is finite: 
$\lVert X \rVert := \Big(\int_I \E \big[ \underset{t \leq s \leq T}{\sup} |X_s^u |^2 \big] \d u  \Big)^{\frac{1}{2}}<\infty.$
    \end{itemize}
       \noindent We say that $(X^u)_{u \in I} \in \Sc_t$ is a solution to \eqref{eq : state equations LQ common noise} if the equations in \eqref{eq : state equations LQ common noise} are satisfied for a.e. $u \in I$. Moreover, we say that the solution is unique if whenever $(X^u)_{u \in I}$ and $(\tilde{X}^u)_{u \in I}$ solve \eqref{eq : state equations LQ common noise}, then the processes $X^u$ and $\tilde{X}^u$ coincide up to a $\P$-null set for a.e. $u \in I$. This is the uniqueness definition used in \cite{de2026mean}. 
\end{Definition}
\end{Definition}
We can now state the following theorem ensuring the well-posedness of the state equation, extending the corresponding result in \cite{de2026linear} to the case with common noise.
\begin{Theorem}\label{thm: existence_unicityX}
    Let $t \in [0,T]$, $\boldsymbol{\xi} \in \Ic_t$ and $\boldsymbol{\alpha} \in \Ac$. Then, there exists a unique solution $\bX=(X^u)_{u \in I}$ to the state equation \eqref{eq : state equations LQ common noise} in the sense of Definition \ref{def : solution space S_t}. Moreover, there exists a measurable function $x : I \times [0,T] \times \Cc_{[0,T]} \times \Cc_{[0,T]} \times \mathcal{K}  \to \R^d$ such that for a.e. $u \in I$, 
    \begin{align}\label{eq : representation_X}
       X_s^u = x(u,s,W^u_{. \wedge s}, \tilde{B}^0_{. \wedge s},R), \quad  \P \text{-a.s}, \quad t \leq s \leq T.
    \end{align}
    In addition, the following estimate holds for some constant $C_T > 0$
      \begin{align}\label{eq :estimate_X}
        \int_I \E \Big[ \underset{t \leq s \leq T}{\sup} |X_s^u |^2 |\tilde{\Fc}_t^0\Big] \d u &\leq C_T \Big( \E \Big[  \int_{t}^{T}  \lVert \beta_s \rVert^2_{L^2(I;\R^d)}+ \lVert \gamma_s \rVert_{L^2(I;\R^d)}^2 \big)  \d s  \big | \tilde{\Fc}_t^0 \Big] \notag \\
        &\qquad \qquad \quad \qquad +  \int_I \E \big[ | \xi^u |^2 |\tilde{\Fc}_t^0\big] \d u + \int_I \int_{t}^{T} \E \big[ |\alpha_s^u |^2 | \tilde{\Fc}_t^0 \big] \d s \d u  \Big), \P \text{-a.s} \ \\
    \end{align}
\end{Theorem}
\begin{proof}
    The proof is postponed to Appendix \ref{appendix: proof_existence_unicityX}.
\end{proof}

\paragraph{Conditional cost functional.} Given $t\in[0,T]$ and $\boldsymbol{\xi} \in\Ic_t$, the goal is  to minimize over all $\balpha \in \Ac$ and state processes  given by \eqref{eq : state equations LQ common noise}, a conditional functional of the form 
\begin{align}\label{eq:J_cost_functional}
    J_{}(t,\boldsymbol{\xi},\balpha) &:= \;  \int_ I\E\Big[  \int_t^T \langle X^u_s,Q_s^u X^u_s \rangle + \langle \bar X^u_s,\int_I G_{s}^Q(u,v)\bar X^v_s\d v\rangle + \langle \alpha^u_s + E_s^u,R_s^u (\alpha^u_s +E_s^u) \rangle\\
    &\qquad \qquad + \;  \langle X^u_T, H^u X^u_T \rangle + \langle \bar X^u_T, \int_I G^{H}(u,v)\bar X^v_T\d v \rangle  \big | \tilde{\Fc}_t^0\Big]\d u.
\end{align}
We refer to the control problem \eqref{eq : state equations LQ common noise}-\eqref{eq:J_cost_functional} as the strong formulation.
We point out that for any $\boldsymbol{\xi}$ and $\balpha \in \Ac$, $\big(J(t,\boldsymbol{\xi},\balpha)\big)_{0 \leq t \leq T}$ is an $\tilde{\F}^0$-adapted process. We now give the assumptions on the model coefficients of the conditional cost functional.
\begin{Assumption}\label{assumption : cumulative_terminal_reward}
\begin{enumerate}
    \item [$\bullet$]  $Q \in \Sc^2_{\tilde{\F}^0}(I;\S^d_+) \cap L^{\infty}(I \ \times \Omega ; \Cc^{d \times d}_{[0,T]})$, $G^Q \in \Sc^2_{\tilde{\F}^0}(L^2(I \times I ; \R^{d \times d}))$ and $T_{G^Q} \in L^{\infty}(\Omega  ; \Cc([0,T]; \Lc(L^2(I;\R^d))$.
    \item [$\bullet$] $R \in \S^2_{\tilde{\F}^0}(\S^m_{> +}) \cap L^{\infty}(I \times \Omega; \Cc^{m \times m }_{[0,T]})$  and $E \in \Sc^2_{\tilde{\F}^0}(I;\R^m) \cap L^{2}(I \times \Omega;\Cc^{m}_{[0,T]})$. Moreover, we assume that there exists a constant $c > 0$ such that 
    \begin{align}\label{ass:posR}
         R_s^u  \succeq c I_m, \quad \d u \otimes \d \P-\text{a.e.},
    \end{align}
    for every $s \in [t,T]$ where $\succeq $ denotes the usual partial order on $\S^m_{>+}$.
    \item [$\bullet$] $H \in \Sc^2_{\tilde{\Fc}_t^0}(I; \S_{+}^d) \cap L^{\infty}(I \times \Omega ; \R^{d \times d})$, $G^H \in \Sc^2_{\tilde{\Fc}_T^0}(I \times I; \R^{d \times d})$ and $T_{G^H} \in L^{\infty}(\Omega; \Lc(L^2(I;\R^d)))$.
     \item [$\bullet$] Moreover, we assume that
\begin{align}\label{eq : inequality_condition}
\begin{cases}
    &\int_I \E \big[    \langle X^u, Q_s^u X^u \rangle + \langle \bar{X}^u , \int_I G_s^Q(u,v) \bar{X}^v \d v \rangle  \big | \tilde{\Fc}_t^0  \big] \geq 0, \quad \P-\text{a.s.}, \text{ for every $s \in [t,T]$.} \\
    &\int_I \E \big[    \langle X^u, H^u X^u \rangle + \langle \bar{X}^u , \int_I G^H(u,v) \bar{X}^v \d v \rangle  \big | \tilde{\Fc}_t^0  \big] \geq 0, \quad \P-\text{a.s.},
\end{cases}
\end{align}
for every $\bX \in \Sc_t$.
\end{enumerate}

\end{Assumption}

Without loss of generality, we assume that $(G_s^Q)^*=G_s^Q$ for every $s \in [t,T]$ and $({G}^H)^* = G^H$ where we recall that the $^\star$ notation has been introduced in  \eqref{eq : tilde_kernel}.

\begin{Remark}\label{rmk : conditional_functional_cost}
    \begin{itemize}
        \item [(1)] The conditional functional cost \eqref{eq:J_cost_functional} is well-defined under Remark \ref{rmk : well_posedness} and following \eqref{eq :estimate_X}, we have for all $t \in [0,T]$, $\boldsymbol{\xi} \in \Ic_t$, $\balpha \in \Ac$
        \begin{small}
        \begin{align}\label{eq : cost_functional_bound}
            J(t,\boldsymbol{\xi},\balpha)  \leq C_T \Big( \E \Big[  \int_{t}^{T}  \lVert \beta_s \rVert^2_{L^2(I;\R^d)}+ \lVert \gamma_s \rVert_{L^2(I;\R^d)}^2 \big)  \d s  \big | \tilde{\Fc}_t^0 \Big] +  \int_I \E \big[ | \xi^u |^2 |\tilde{\Fc}_t^0\big] \d u + \int_I \int_{t}^{T} \E \big[ |\alpha_s^u |^2 | \tilde{\Fc}_t^0 \big] \d s \d u  \Big), \P \text{-a.s}.
        \end{align} 
        \end{small}
        \item [(2)] By Assumption \ref{assumption : cumulative_terminal_reward}, we have for every $t \in [0,T]$, $\boldsymbol{\xi} \in \Ic_t$ and $\balpha \in \Ac$ 
        \begin{align}
            J(t,\boldsymbol{\xi},\balpha) \geq 0  \quad \P \text{-a.s. }.
        \end{align}
        \item [(3)] The inequalities  \eqref{eq : inequality_condition} can be obtained under the stronger assumption that $T_{G_s^Q}$ is  a $\P- \text{a.s.}$ non-negative operator for every $s \in [t,T]$ and $T_{G^H}$ is a $\P-\text{a.s.}$ non-negative operator where the non-negativity property for operators has been defined in Section \ref{sec:notations}.

    \end{itemize}
\end{Remark}
\normalsize

\begin{Remark}\textnormal{(A centered formulation of the control problem).}\label{rmk : alternative_formulation_centered}
As in \cite[Appendix B]{de2026linear}, we can formulate the control problem in a centered way as follows
    \begin{align}\label{eq : centered_formulation}
        \hat{J}(t,\boldsymbol{\xi},\balpha) &= \int_I \E \Big[ \int_{t}^{T} \big\langle Q_s^u (X_s^u - \int_I \hat{G}_s^Q(u,v) \bar{X}_s^v ), X_s^u - \int_I \hat{G}_s^Q(u,v) \bar{X}_s^v )  \big \rangle  + \langle \alpha_s^u + E_s^u , R_s^u (\alpha_s^u +E_s^u) \rangle   \notag \\
          &\qquad \qquad + \;  \big \langle H^u  (X^u_T - \int_I \hat{G}^H(u,v) \bar{X}_T^v ), X_T^u - \int_I \hat{G}^{H}(u,v)\bar X^v_T\d v  \big\rangle \big | \tilde{\Fc}_t^0 \Big]\d u, 
    \end{align}
    for model coefficients satisfying Assumption \ref{assumption : cumulative_terminal_reward}.
     They show that the formulation \eqref{eq : centered_formulation} is equivalent to the formulation \eqref{eq:J_cost_functional} after defining the kernels $G^Q$ and $G^H$ as follows
    $ G_s^Q(u,v) = \big(\hat{G}_s^Q \circ Q_s \circ \hat{G}_s^Q \big)(u,v)  - Q_s^u \hat{G}_s^Q(u,v) - Q_s^v \hat{G}_s^Q(u,v), $ $
        G^H(u,v) = \big( \hat{G}^H \circ H \circ \hat{G}^H)(u,v) - H^u \hat{G}^H(u,v) - H^v \hat{G}^H(u,v).$
    With such a choice of kernels $G^Q$ and $G^H$, we therefore have that Equation \eqref{eq : inequality_condition} is satisfied.
\end{Remark}

\begin{Remark}[Label-state formulation of the control problem]\label{rmk : weak_formulation_control_problem}
The problem could have been formulated in a label-state formulation  \cite{lacker2023label} which avoids  the technical measurability issues of dealing with the collection of Brownian motions $\big \lbrace W^u : u \in I \big \rbrace $, namely the joint measurability over the product $I \times \Omega$ of the processes $\big \lbrace X^u : u \in I \big \rbrace$ and Remark \ref{rmk : well_posedness}. In this setting, one would work on a completed probability space $(\Omega,\Fc,\P)$ supporting a $\R$-valued standard Brownian motion, a uniform random variable $U$ over $I$ used for the labeling of the agents and $R$ a uniform random variable over $\Kc$ used for the randomization of the initial condition.  The associated controlled state process would be then given by
\begin{align}\label{eq : weak_version_control_problem}
\begin{cases}
        & \d X_s =\left [ \beta_s(U) + A_s(U) X_s +  \tilde{\E}^0_{(\tilde{U},\tilde{X}_t) \sim \P_{(U,X_t)}}  \big[G_{s}^A(U,\tilde{U}) \tilde{X}_s \big] + B_s(U) \alpha_s \right] \d s \\
    &\qquad\qquad +\left [\gamma_s(U) + C_s(U)  X_s +    \tilde{\E}^0_{(\tilde{U},\tilde{X}_t) \sim \P_{(U,X_t)}} \big[ G_{s}^C(U,\tilde{U}) \tilde{X}_s \big] + D_s(U) \alpha_s \right ] \d W_s+\theta_s(U)  \d \tilde{B}_s^0,\quad s\in[t,T], \\ 
    &X_0 = \xi,
\end{cases}
\end{align}
where we assumed that $\mathbb{P}_{\xi^u} = \P_{\xi | U=u}$ $\lambda(\d u )-\text{a.e.}$ or equivalently that $\mathbb{P}_{(U,\xi)}(\d u ,\d x) = \mathbb{P}_{\xi^u}(\d x) \lambda(\d u)$ as measures over the space $\Pc_2^{\lambda}(I \times \R^d) := \big \lbrace \mu \in \Pc_2(I \times \R^d) : \text{pr}_{1} \sharp \mu = \lambda \big \rbrace$ where $\text{pr}_1$ denotes the projection map over $I$. Moreover, we refer to $\hat{\tilde{\E}}^0$ as the conditional expectation with respect to $\hat{\tilde{B}}^0$ over a copy space $(\hat{\Omega}, \hat{\Fc},\hat{\P})$ on which are defined $(\hat{U},\hat{X}, \hat{\tilde{B}}^0)$ the independent copies of $(U,X,B^0)$. The cost functional would be given by
\begin{align}\label{eq : cost_functional_weak}
    J(t,\xi,\alpha) &= \E \Big[ \int_{t}^{T} \langle X_s  , Q_s(U) X_s \rangle + \hat{\tilde{\E}}^0_{(\hat{U},\hat{X}_s) \sim \P_{(U,X_s)}}[ X_s G_s^Q(U,\hat{U}) \hat{X}_s] + \langle \alpha_s + E_s(U), R_s(U) (\alpha_s + E_s(U) \rangle \\
    &\quad \quad + \langle X_T, H(U) X_T \rangle + \hat{\tilde{\E}}^0_{(\hat{U},\hat{X}_T) \sim \P_{(U,X_T)}}[X_T G^H(U,\hat{U}) \hat{X}_T] \big | \tilde{\Fc}_t^0 \Big]
\end{align}
The control problem \eqref{eq : weak_version_control_problem}-\eqref{eq : cost_functional_weak} can be viewed as a standard McKean-Vlasov control problem over the extended state space $I \times \R^d$ which allows a richer analysis of the interacting particle system. Whenever we are interested in a cost functional involving the law of the joint process $(U,X)$, \eqref{eq : weak_version_control_problem} can be a good substitute, mathematically speaking, to \eqref{eq : state equations LQ common noise}.
In this setting, we introduced  the filtration $\F = (\Fc_t)_{0 \leq t \leq T}$ generated by the random variables $(U,W,R)$ completed by the family of $\P$-null sets and by $\F^0 := (\Fc_t^0)_{0 \leq t \leq T}$ the filtration generated by $(U,W,\tilde{B}^0,R)$ completed with the family of $\P$-null sets. In this setting, the admissibility conditions are
\begin{align}
\begin{cases}
    \xi &= \xi(U,t, W_{. \wedge t}, \tilde{B}^0_{. \wedge t}, R), \quad \P-\text{a.s.} \text{ and } \E[ |\xi|^2] < \infty, \\
    \alpha_t &= \alpha(U,t, W_{. \wedge t}, \tilde{B_0}_{. \wedge t}, R), \quad \P-\text{a.s and } \E \Big[ \int_{0}^{T} | \alpha_s|^2 \d s \Big]  < \infty,
\end{cases}
\end{align}
for some Borel maps $\xi$ and $\alpha$. We point out that  this setting requires $\hat{\Omega} \ni \hat{\omega} \mapsto \P_{(\hat{\tilde{U}},\hat{\tilde{X_s}})|\hat{\tilde{B_0}}=\hat{\tilde{B_0}}(\hat{\omega})}$ which is a consequence of the disintegration theorem over the Polish space $I \times \R^d \times \Cc_{[0,T]}$. The strong and label-state formulations are mathematically equivalent in the sense that they induce the same distribution over the extended state space $I\times\mathbb{R}^d$. However, they emphasize different aspects of the model. The label-state formulation naturally fits within the classical McKean--Vlasov framework and circumvents the technical issues related to the joint measurability of the continuum of processes $(X^u)_{u\in I}$. On the other hand, the strong formulation preserves the pathwise interpretation of the dynamics, allowing one to track the evolution of each labeled agent $u$ individually. This feature is particularly useful when studying agent-wise properties or deriving pathwise estimates, and motivates our choice of the strong formulation throughout this work.
\end{Remark}

\section{Solution  of the control problem}\label{sec:solution}

In this subsection, we work on the control problem formulated in \eqref{eq : state equations LQ common noise}-\eqref{eq:J_cost_functional}.  We are going to introduce a new system of Backward Stochastic Riccati equation (BSRE) taking values  in Hilbert spaces related to our control problem. This will be used to derive a \emph{fundamental relation} for our linear-quadratic control problem.

\subsection{Infinite dimensional system and fundamental relation}\label{subsec : infinite_dimensional_system}

\noindent \textbf{Standard BSRE.}

\noindent For $\text{a.e.}$  $u \in I$, we introduce the  following backward Riccati equation
\begin{align}\label{eq : standardRiccati K}
     \d K_t^u =  -  \tilde{\Phi}^u(t,K_t^u) \d t  + Z_t^{K^u} \d W_t^0, \quad 0 \leq t \leq T,  
    \quad K_T^u = H^u,
\end{align}
where the random map $\tilde{\Phi}^u  : \Omega \times [0,T] \times \S^d_+ \ \to \R^{d \times d}$ is defined by
\begin{align}\label{eq : def_tilde_phi_u}
    \tilde{\Phi}^u(t,k ) = \Phi^u(t,k) - U_t^u(k)^{\top} O_t^u(k) U_t^u(k),
\end{align}
and where we defined the random maps 
\begin{align}\label{eq : notations_common_noise_Kt_no_control_vol}
\begin{cases}
&\Phi^u(t,k) :=  (A_t^u)^{\top}k+ k A_t^u  + (C_t^u)^{\top} k C_t^u + Q_t^u, \quad 
O_t^u(k):= R_t^u +  (D_t^u)^{\top} k D_t^u ,  \\
&U_t^u(k) := (B_t^u)^{\top} k +(D_t^u)^{\top} k C_t^u.
\end{cases}
\end{align}
We note that for  $\text{ a.e. } u \in I$, the BSRE \eqref{eq : standardRiccati K} is similar to those studied in  \cite{basei2019weak} and \cite{pham2016lqcmkv}. 

\begin{Remark}\label{remark:invO}
 Observe that, for all $\kappa\in\S^d_+$ and for every $t \in [0,T]$, we have  $O_t(\kappa) \in L^{\infty }(I \times \Omega; \S^d_{+})$ and  the inverse is well-defined $\d u \otimes \d \P-\text{a.e}$. Moreover, following \eqref{ass:posR}, we  have   $|O^u_t(\kappa)^{-1} | \leq M,\quad \d u \otimes \d \P-\text{a.e.}$,  for some positive constant $M$ independent of $u$ and for every $t \in [0,T]$.
\end{Remark}

\begin{Definition}\label{def : definition_solution_K}
    We define a solution to \eqref{eq : standardRiccati K} as a pair of collections of processes  $(K,Z^K)=(K^u,Z^{K^u})_{u \in I}$ such that  for a.e. $u \in I$, $(K^u,Z^{K^u})$ solves the equation \eqref{eq : standardRiccati K} and $(K,Z^K) \in  \Sc^2_{\tilde{\F}^0}(I;\R^{d \times d}) \times \Hc^2_{\tilde{\F}^0}{(I;\R^{d \times d})}$.

We say that a solution is unique, if whenever $(K^u,Z^{K^u})_{u \in I}$ and $(\hat{K}^u,Z^{\hat{K}^u})_{u \in I}$ are solutions to \eqref{eq : standardRiccati K}, then the pair of processes $(K^u,Z^{K^u})$ and $(\hat{K}^u,Z^{\hat{K}^u})$ coincide up to a $\P$-null set, for almost every $u \in I$.
\end{Definition}

To ease the notations, we will  denote 
\begin{align}
   \Phi^u(t,K_t^u) := \Phi_t^u, \quad 
   O_t^u(K_t^u) := O_t^u, \notag \quad
   U_t^u(K_t^u) := U_t^u.  
\end{align}

\noindent \textbf{Abstract Backward Stochastic Riccati equation.}

Given a solution $(K,Z^K) =(K^u,Z^{K^u})_{u \in I}$ to \eqref{eq : standardRiccati K}, we introduce the following Riccati BSDE on the Hilbert space $L^2(I \times I;\R^{d \times d})$
\begin{align}\label{eq : abstract Riccati barK}
    \d \bar{K}_t =- F(t,\bar{K}_t) \d t + Z_t^{\bar{K}} \d W_t^0, \quad 0 \leq t \leq T, \quad 
   \bar{K}_T = G^H,
\end{align}
where the random map $F : \Omega \times [0,T] \times L^2(I \times I ;\R^{d \times d}) \to L^2(I \times I;\R^{d \times d})$ is defined $\d u \otimes \d v$
a.e. by 
\begin{align}\label{eq : properties_functionF}
    F(t,\bar{k})(u,v) &:= \Psi(u,v)(t,K_t,Z_t^K, \bar{k}) - (U_t^u)^{\top} (O_t^u)^{-1} V(u,v)(t,K_t,\bar{k})  - V(u,v)^{\star}(t,K_t,\bar{k}) (O_t^v)^{-1} U_t^v  \notag \\
    &\quad - \big( V^*(t,K_t,\bar{k}) \circ O _t^{-1} \circ V(t,K_t,\bar{k}) \big)(u,v),   \\
\end{align}
for every $t \in [0,T]$, $\bar{k} \in L^2(I \times I;\R^{d \times d})$ and where we define 
\begin{align}\label{eq : notations_barKt_common_noise_no_control_vol}
\begin{cases}
    \Psi(u,v)(t,K_t, Z_t^K,\bar{k}) &:= 
    K_t^u  G_t^A(u,v)+ (G_t^A)^{\star}(u,v) K_t^v   + (C_t^u)^{\top} K_t^u G_t^C(u,v)   \\
    &\quad + (G_t^C)^{\star}(u,v) K_t^v C_t^v 
     + \big((G_t^C)^* \circ K_t \circ G_t^C \big)(u,v)+  \\
    &\quad   + (A_t^u)^{\top} \bar{k}(u,v) + \big((G_t^A)^* \circ \bar{k} \big)(u,v) + \bar{k}^{\star}(u,v) A_t^v  \\
    &\quad + \big( \bar{k}^* \circ G_t^A)(u,v) +G_t^Q(u,v) 
     \\
    V(u,v)(t,K_t,\bar{k}) &:=(D_t^u)^{\top} K_t^u G_t^C(u,v)  + (B_t^u)^{\top} \bar{k}(u,v),    \\
\end{cases}
\end{align}
where we recall the $\circ$ composition introduced in \eqref{eq : kernel_product} and \eqref{eq : kernel_form_multiplicativity} and the $^*$ notation introduced in \eqref{eq : tilde_kernel}.

\normalsize

\begin{Lemma}\label{lemma : prop_randommap_F} The random map $F$ satisfies the following: 
\begin{enumerate}
    \item [$(1)$] for any $\bar{k} \in L^2(I \times I ; \R^{d \times d})$, there exists a positive constant $C_T^1$ such that
    \begin{align}
        \lVert T_{F(t,\bar{k})} \rVert_{\Lc(L^2(I;\R^d)} \leq C_T^1 \big( 1+ \lVert T_{\bar{k}} \rVert_{\Lc(L^2(I;\R^d))} + \lVert T_{\bar{k}} \rVert^2_{\Lc(L^2(I;\R^d))} \Big), \quad \P-\text{a.s.}, \quad 0 \leq t \leq T.
    \end{align}
    \item [$(2)$] For any $\bar{k}$, $\bar{k}' \in L^2(I \times I ; \R^{d \times d})$, there exists a positive constant $C > 0$ such that
    \begin{align}
        \lVert F(t,\bar{k}') - F(t,\bar{k}) \rVert_{L^2(I \times I ; \R^{d \times d})} \leq C \Big( (1+ \lVert T_{\bar{k}} \rVert_{\Lc(L^2(I ;\R^d))} + \rVert T_{\bar{k}'} \rVert_{\Lc(L^2(I;\R^d))} ) \lVert \bar{k'} - \bar{k} \rVert_{L^2(I \times I ; \R^{d} \times \R^d)} \Big).
    \end{align}
    \item [$(3)$] The random map $F$ is $\P-\text{a.s.}$ continuous on $[0,T] \times L^2(I \times I ; \R^{d \times d})$.
    \item [$(4)$] $F(t,\bar{k})^{\star}(u,v) = F(t,\bar{k})(u,v)$, for all $t \in [0,T]$, $\bar{k} \in L^2(I \times I ; \R^{d \times d})$ and $(u,v) \in I \times I$.
\end{enumerate}
\end{Lemma}
\begin{proof}
    We prove $(1)$. Let $\bar{k} \in L^2(I \times I ;\R^{d \times d})$ and $t \in [0,T]$. By definition of the random map $F(t,\bar{k})$ in \eqref{eq : properties_functionF}, we use the fact that the operator $T$ defined in \eqref{eq : operator_definition} is linear in its variable $\bar{k}$ and we recall the sub-multiplicativity of the operator norm in \eqref{eq : sub_multiplicativity_operator_norms_with_multiplication} after noticing that $\lVert T_{V(t,\bar{K}_t, \bar{k})} \rVert_{\Lc(L^2(I;\R^d))} \leq C \big( 1+  \lVert T_{\bar{k}} \rVert_{\Lc(L^2(I;\R^d))} \big), \quad \P-\text{a.s.}$ for some constant $C \geq 0$. Hence, using again the sub-multiplicativity of the operator norm in \eqref{eq : sub_multiplicativity_operator_norms_with_multiplication}, and the assumptions on all the model coefficients,  the result follows where the quadratic term $\lVert T_{V^{\star}(t,K_t, \bar{k})} \circ M_{O_t^{-1}} \circ T_{V(t,K_t,\bar{k})} \rVert_{\Lc(L^2(I;\R^d))}$ is handled recalling \eqref{eq : operator_adjoint},\eqref{eq : sub_multiplicativity_operator_norms_with_multiplication}, \eqref{eq : notations_common_noise_Kt_no_control_vol}.
    \\
    We prove $(2)$. Let $k,k' \in L^2(I \times I ;\R^{d \times d})$ and $t \in [0,T]$. We focus on the quadratic term as it is the main difficulty. In fact, it is enough to consider the map 
    \begin{align}
        (u,v) \mapsto \int_I \bar{k}(u,w)  \bar{k}(w,v) \d w  - \int_I \bar{k}'(u,w) \bar{k}'(w,v) \d w,
    \end{align}
    where we omit the terms $(B,O)$ since we have by assumption $\underset{(u,\omega) \in I \times \Omega}{\text{ ess sup }} \underset{0 \leq t \leq T}{\text{ sup }} |O_t^u(\omega)^{-1}| < \infty$ and similarly for $B$. Therefore, we have
    \begin{small}
    \begin{align}\label{eq : equality_quadratic_terms}
        \int_I \bar{k}(u,w)  \bar{k}(w,v) \d w  - \int_I \bar{k}'(u,w) \bar{k}'(w,v) \d w &= \int_I \bar{k}(u,w)  \big( \bar{k}(w,v) - \bar{k}'(w,v) \big) \d w + \int_I \big(\bar{k}(u,w)- \bar{k}'(u,w) \big) \bar{k}'(\omega,v)  \d w,
    \end{align}
    \end{small}
    Fix $v \in I$ and denote the map $H_v(w) = \bar{k}(w,v) - \bar{k}'(w,v)$. We notice that
    \begin{align}
        \int_I \bar{k}(u,w) \big(\bar{k}(w,v) - \bar{k}'(w,v) \big) \d w = \tilde{T}_{\bar{k}} H_v(u),
    \end{align}
    where $\tilde{T}_{\bar{k}}$ is the extension of $T_{\bar{k}}$ over $L^2(I;\R^{d \times d })$, i.e. 
    \begin{align}
        L^2(I: \R^{d \times d}) \ni H \mapsto \tilde{T}_{\bar{k}} H(\cdot) := \int_{I} \bar{k}(\cdot,v) H(v) \d v \in L^2(I;\R^{d \times d}).
    \end{align}
    Moreover, we have $\lVert \tilde{T}_{\bar{k}} \rVert_{\Lc(L^2(I ; \R^{d \times d}))} = \lVert T_{\bar{k}} \rVert_{\Lc(L^2(I;\R^d)}$. Indeed, let $H \in L^2(I;\R^{d \times d})$ that we rewrite $H=(H_1,\ldots, H^d)$ where $H^i \in L^2(I;\R^d)$ for any $1 \leq i \leq d$. Then, we have that $\tilde{T}_{\bar{k}} H= (T_{\bar{k}} H_1,\ldots, T_{\bar{k}} H^d)$. Hence, it follows that
        \begin{align}
        \lVert \tilde{T}_{\bar{k}} H \rVert_{L^2(I;\R^{d \times d})} = \sum_{j=1}^{d} \lVert T_{\bar{k}} H^j \rVert_{L^2(I;\R^d)} &\leq \lVert T_{\bar{k}} \rVert_{\Lc(L^2(I;\R^d))} \sum_{j=1}^{d} \lVert H^j \rVert_{L^2(I;\R^d)}, \\
        &= \lVert T_{\bar{k}} \rVert_{\Lc(L^2(I;\R^d))} \lVert H \rVert_{L^2(I; \R^{d \times d})},
    \end{align}
    from which we deduce $\lVert \tilde{T}_{\bar{k}} \rVert_{\Lc(L^2(I;\R^{d \times d}))} \leq \lVert T_{\bar{k}} \rVert_{\Lc(L^2(I; \R^d))}$. Moreover, the reverse inequality is obtained by considering $H=(f,0,\ldots,0) \in L^2(I;\R^{d \times d})$ for any $f \in L^2(I;\R^d)$. Therefore, the claim follows.
    Hence, by definition of the operator norm, it follows that
    $\lVert \tilde{T}_{\bar{k}} H_v \rVert_{L^2(I;\R^{\d \times d}) } \leq \lVert \tilde{T}_{\bar{k}} \rVert_{\Lc(L^2(I ;\R^{\d \times d})} \lVert H_v \rVert_{L^2(I; \R^{d \times d})}$. Now, noticing that 
    \begin{align}
        \lVert \int_I  \bar{k}(\cdot , w) \big( H_{\cdot}(w) \big) \d w  \rVert^2_{L^2(I \times I;\R^{d \times d})} &= \lVert \tilde{T}_{\bar{k}} H_{\cdot}(\cdot) \rVert_{L^2(I \times I; \R^{d \times d})}^2 = \int_I \int_I  |\tilde{T}_{\bar{k}} H_v(u)|^2 \d u \d v
        = \int_I \lVert \tilde{T}_{\bar{k}}H_v \rVert^2_{L^2(I; \R^{d \times d})} \d v , \\
        &\leq  \rVert \tilde{T}_{\bar{k}} \rVert_{\Lc(L^2(I ;\R^{\d \times d}))}^2 \int_I \lVert H_v \rVert_{L^2(I; \R^{d \times d})}^2 \d v 
        =  \rVert T_{\bar{k}} \rVert_{\Lc(L^2(I ;\R^{\d \times d})}^2 \lVert \bar{k}' - \bar{k} \rVert^2_{L^2(I \times I ; \R^{d \times d})}.
    \end{align}
    For the second term in the right-hand side of \eqref{eq : equality_quadratic_terms}, we use the adjoints. Indeed, we end up doing similar computations to the previous case by denoting the map $(u,v) \mapsto L^{\star}(v,u) =L(u,v)^{\top}= \int_I \bar{k}'(w,v)^{\top} \big( \bar{k}(u,w) - \bar{k}'(u,w) \big)^{\top} \d w= \int_I \bar{k}^{',\star}(v,w) \big( \bar{k}^{\star}(w,u) - (\bar{k}^{'})^{\star}(w,u) \big) \d w $. Hence, it follows that
    \begin{align}
       \rVert  L^{\star} \rVert_{L^2(I \times I;\R^{d \times d})} \leq \lVert T_{(\bar{k}')^{\star}} \rVert_{\Lc(L^2(I;\R^{d \times d}))} \lVert \bar{k}^{\star} - (\bar{k}^{'})^{\star} \rVert_{L^2(I \times I ; \R^{d \times d})}.
     \end{align}
     Now, using that $\lVert T_{(\bar{k}^{'})^{\star}} \rVert_{\Lc(L^2(I ; \R^d))} = \lVert T_{\bar{k}'} \rVert_{\Lc(L^2(I;\R^{d}))}$ and that $\lVert \bar{k}^{\star} - (\bar{k}^{'})^{\star} \rVert_{L^2(I \times I ; \R^{d \times d})} = \lVert \bar{k} - \bar{k}' \rVert_{L^2(I \times I ; \R^{d \times d})}$, we therefore end up with
     \begin{align}
         \lVert  \int_I \bar{k}(\cdot,w)  \bar{k}(w,\cdot) \d w  - \int_I \bar{k}'(\cdot,w) \bar{k}'(w,\cdot) \d w \rVert_{L^2(I \times I ; \R^{d \times d})} \leq  \big( \lVert T_{\bar{k}}\rVert_{\Lc(I;\R^d)} + \lVert T_{\bar{k}'} \rVert_{\Lc(L^2(I;\R^d)} \big) \big \lVert \bar{k}- \bar{k}' \rVert_{L^2(I \times I ; \R^{d \times d})}
     \end{align}
     \\
     We prove $(3)$. The $\P-\text{a.s.}$ continuity of the map follows from the assumption on the model coefficients and the fact that the solution $K$ from \eqref{eq : standardRiccati K}  is itself  continuous.
     \\
     We prove $(4)$. We show the property for the term $\Psi$ since it is clear that the other terms satisfy the invariance by the adjoint operation. Hence, by definition of $\Psi$ in \eqref{eq : notations_barKt_common_noise_no_control_vol} and by the symmetry of the solution $K$ to the standard BSRE, the assumption on $G^Q$ and the fact that $(G_1 \circ G_2)^{\star} = (G_2^{\star} \circ G_1^{\star})$, the result follows.
\end{proof}

\begin{Definition}\label{def : eq_abstracticatti barK}
    We define a solution to \eqref{eq : abstract Riccati barK} as a pair of processes $(\bar{K}, Z^{\bar{K}})$ such that $(\bar{K},Z^{\bar{K}})$ solves equation \eqref{eq : abstract Riccati barK} and $(\bar{K}, Z^{\bar{K}}) \in \S^2_{\tilde{\F}^0} \big(L^2(I \times I, \R^{d \times d}) \big) \times \H^2_{\tilde{\F}^0} \big(L^2(I \times I,\R^{d \times d} ) \big)$.
    
We say that a solution is unique if whenever $\bar{K}$ and $\tilde{\bar{K}}$ are solutions to \eqref{eq : abstract Riccati barK}, then the $L^2(I \times I;\R^{d \times d})$-valued processes $(\bar{K},Z^{\bar{K}})$ and $(\tilde{\bar{K}},Z^{\tilde{\bar{K}}})$ coincide up to a $\P$-null set. 
\end{Definition}
\begin{Remark}
\begin{enumerate}
    \item [(1)] In fact, we will show that the operator $T_{\bar{K}} \in L^{\infty}(\Omega;\Cc([0,T];\Lc(L^2(I;\R^d)))$ which implies the existence of a positive constant C such that
    \begin{align}\label{eq : a_priori_estimate_control}
        \underset{s \in [t,T]}{\sup} \lVert T_{\bar{K}_s} \rVert_{\Lc(L^2(I;\R^d))} \leq C, \quad  \P-\text{a.s.}.
    \end{align}
    \item [(2)] Note that if $(\bar{K}, Z^{\bar{K}})$ is a solution under Definition \ref{def : eq_abstracticatti barK}, then there exists Borel measurable functions, $\bar{k}$ and $\bar{z}$ defined on $I \times I \times [0,T] \times \Cc_{[0,T]}$ such that
            $\bar{K}_t(u,v) = \bar{k}(u,v,t, \tilde{B}^0_{. \wedge t}), $ $ \bar{Z}_t(u,v) = \bar{z}(u,v,t,\tilde{B}^0_{. \wedge t}), $ $ \d u \otimes \d v \otimes \d \P  \text{ a.e.}.$
\end{enumerate}
\end{Remark}
\noindent To ease the notations, we will now denote 
\begin{align}
   \Psi(u,v)(t,K_t,Z_t^K, \bar{K}_t) := \Psi_t(u,v), \quad
   V(u,v)(t,K_t,\bar{K}_t) := V_t(u,v).
\end{align}
\noindent \textbf{Linear Equations.} Given $(K,Z^{K})$ and $(\bar{K},Z^{\bar{K}})$ respectively solutions to \eqref{eq : standardRiccati K} and \eqref{eq : abstract Riccati barK}, we introduce the following Riccati BSDE valued in the Hilbert space $L^2(I; \R^d)$
\begin{align}\label{eq : linear_Riccati_Y_common_noise}
    \d Y_t  = - \hat{F}(t,Y_t) \d t + Z_t^{Y} \d W_t^0,  \quad 
    Y_T = 0,
\end{align}
where the random map $\hat{F} : \Omega \times [0,T] \times  L^2(I;\R^d) \to L^2(I; \R^d)$ is defined $\d u \text{ a.e.}$ by 
\begin{align}\label{eq :  tilde F}
    \hat{F}(t,y)(u) &:=M(u)(t,K_t, Z_t^K, \bar{K}_t, \bar{Z}_t^{K},y)
     - (U_t^u)^{\top} (O_t^u)^{-1} \Gamma^u(t,K_t,y) -  T_{V_t^*}\big((O_t)^{-1} \Gamma(t,K_t,y) \big) 
\end{align}
for every $t \in [0,T]$, $y,z \in L^2(I;\R^d)$, and where we defined 
\begin{align}\label{eq : notations_common_noise_Yt_no_control_vol}
\begin{cases}
    M(u)(t,K_t, Z_t^K, \bar{K}_t, \bar{Z}_t^{K},y) &:=  K_t^u \beta_t^u +  Z_t^{K^u} \theta_t^u +  (C_t^u)^{\top} K_t^u \gamma_t^u   + T_{(G_t^C)^*}(K_t \gamma_t)(u) \\&\quad 
     +  T_{\bar{K}_t}(\beta_t)(u) + T_{Z_t^{\bar{K}}}(\theta_t)(u) + T_{(G_t^A)^*}(y)(u)        \\
    \Gamma(u)(t,K_t,y) &:= (D_t^u)^{\top}  K_t^u \gamma_t^u  + (B_t^u)^{\top} y^u + R_t^u I_t^u ,  
\end{cases}
\end{align}

\begin{Remark}\label{rmk : terms_Y}
\begin{enumerate}
    \item The driver $\hat{F}$ in \eqref{eq : linear_Riccati_Y_common_noise} can be rewritten as 
    \begin{align}\label{eq : driver_tildeF_second_form}
        \hat{F}(t,y) = A_t y + f_t,
    \end{align}
    where for any $t \in [0,T]$, we define the linear operator $A_t$  on $L^2(I;\R^d)$ and $f_t \in L^2(I;\R^d)$ as 
    \begin{align}\label{eq : def_operators}
         \begin{cases}
        L^2(I;\R^d) \ni f \mapsto A_t(f)(u) &= P_t^u f(u) + T_{(V'_t)^*}(f)(u) 
        \\
        I \ni u  \mapsto f_t(u) &=   K_t^u \beta_t^u +  Z_t^{K^u} \theta_t^u +  (C_t^u)^{\top} K_t^u \gamma_t^u   - (U_t^u)^{\top} (O_t^u)^{-1} (D_t^u)^{\top} K_t^u \gamma_t^u \\
        &\quad - (U_t^u)^{\top} (O_t^u)^{-1}R_t^u I_t^u  +T_{(G_t^C)^*}(K_t \gamma_t)(u)  +  
       T_{\bar{K}_t}(\beta_t)(u)  \\
     & \quad +T_{Z_t^{\bar{K}}}(\theta_t)(u) - T_{(V_t)^*}((O_t)^{-1}D_t^{\top} K_t\gamma_t))(u) - T_{(V_t)^*}(O_t^{-1}R_t I_t)(u),
    \end{cases}
    \end{align}
    where to ease notations we have denoted 
    \begin{align}
         V'_t(u,v) = G_t^A(u,v)  - B_t^u (O_t^u)^{-1} V_t(u,v), \quad
         P_t^u = (A_t^u)^{\top} - (U_t^u)^{\top} (O_t^u)^{-1} (B_t^u)^{\top}.
    \end{align}
    \item There exists a positive constant $C$ such that $\underset{0 \leq t \leq T}{\sup} \lVert A_t \rVert_{\Lc(L^2(I;\R^d))} \leq C, \quad \P \text{-a.s}$. 
    
    It follows from the classical inequality $\lVert x + y \rVert^2_{L^2(I;\R^d)} \leq 2 \big ( \lVert x \rVert^2_{L^2(I;\R^d)} + \lVert y \rVert^2_{L^2(I;\R^d)} \big)$, assumption on model coefficients Remark \ref{remark:invO}, Proposition \ref{prop : Ricatti_K} and  Equation
    \eqref{eq : a_priori_estimate_control}.

  \item We have $\E \Big[ \int_{0}^{T} \lVert f_t \rVert_{L^2(I;\R^d)}^2 \d t \Big] < + \infty.$

  It follows again from the assumptions on the model coefficients and Equations \ref{remark:invO},  Proposition \ref{prop : Ricatti_K} and  Equation
    \eqref{eq : a_priori_estimate_control}. 

\end{enumerate}
\end{Remark}

To ease the notations, we will now denote
\begin{align}
   M(u)(t,K_t,Z_t^K, \bar{K}_t, Z_t^{\bar{K}},Y_t):= M_t^u, \quad 
\Gamma(u)(t,K_t,Y_t) := \Gamma_t^u.
\end{align}
\begin{Definition}\label{def : linear_Riccati_Y}
    We define a solution to \eqref{eq : linear_Riccati_Y_common_noise} as a pair of processes $(Y,Z^{Y})$ such that $(Y,Z^{Y})$ solves equation \eqref{eq : linear_Riccati_Y_common_noise} and   $(Y,Z^{Y}) \in \S_{\tilde{\F}^0}^2 \big(L^2(I,\R^d) \big) \times \H_{\tilde{\F}^0}^2 \big(L^2(I,\R^d)\big)$.
    
    We say that a solution is unique, if whenever $(Y,Z^{Y})$ and $(\tilde{Y},Z^{\tilde{Y}})$ are solutions to \eqref{eq : linear_Riccati_Y_common_noise}, then the $L^2(I  ;\R^{d})$ valued processes $(Y,Z^{Y})$ and $(\tilde{Y}, Z^{\tilde{Y}})$ coincide up to a $\P$-null set.
\end{Definition}
\noindent Finally, we introduce for $a.e. $ $u \in I$ the linear BSDE
\begin{align}\label{eq : RiccatiLambda}
    \d \Lambda_t^u = - \tilde{L}_t^u \d t + Z_t^{\Lambda^u} \d W_t^0,   \quad 
    \Lambda_T^u =0,
\end{align}
where we define for a.e. $u \in I$ and for  every $t \in [0,T]$,
\begin{align}\label{eq :  tildeL}
    \tilde{L}_t^u =  L_t^u(K_t,\bar{K}_t,Y_t,Z_t^{Y}) - \langle \Gamma_t^u, (O_t^u)^{-1} \Gamma_t^u \rangle,
\end{align}
with
\begin{align}\label{eq : coeff_Lambda}
    L_t^u(K_t,\bar{K}_t,Y_t,Z_t^{Y}) &:= \langle \gamma_t^u, K_t^u \gamma_t^u \rangle + \langle \theta_t^u, K_t^u, \theta_t^u \rangle + 2 \langle \beta_t^u,Y_t(u) \rangle  + 2 \langle \theta_t^u, Z_t^Y(u) \rangle  \\
    &\quad \quad+\int_I \langle \theta_t^u, \bar{K}_t(u,v) \theta_t^v \rangle \d v  + \langle E_t^u, R_t^u E_t^u \rangle .
\end{align}
\begin{Remark}
    Note that given solutions $(K,Z^K) $, $(\bar{K}, Z^{\bar{K}})$ and $(Y,Z^{Y})$ in the sense respectively of Definition \ref{def : definition_solution_K}, \ref{def : eq_abstracticatti barK} and \ref{def : linear_Riccati_Y}, for a.e. $u \in I$,  $(\Lambda^u,Z^{\Lambda^u})$ is a standard BSDE with given coefficients which can be uniquely solved in $\S^2_{\tilde{\F}^0}(\R^d) \times \H^2_{\tilde{\F}^0}(\R^d)$ from standard theory and we have 
    \begin{align}\label{eq : lambda_representation_common_noise}
        \Lambda_t^u &= \int_{t}^{T} \E \Big[ L_s^u(K_s,\bar{K}_s, Y_s,Z_s^{Y}) - \langle \Gamma_s^u, (O_s^u)^{-1} \Gamma_s^u \rangle \big | \tilde{\Fc}_t^0 \Big ] \d s.
    \end{align}
    which holds $\d u \otimes \d \P$ for every $t \in [0,T]$. 
    Moreover, for every $t \in [0,T]$, the joint measurability of $(u,t,\omega) \mapsto \Lambda_t^u(\omega)$  follows from the suitable representation of the solutions $(K,\bar{K}, Y,Z^{Y})$ and from Remark \ref{rmk : well_posedness}.
\end{Remark}
To ease the notations, we will now denote
\begin{align}
   L_t^u(K_t,\bar{K}_t, Y_t,Z_t^{Y}):= L_t^u.
\end{align}
\begin{Remark}\label{rmk : system_backward_equations}
\begin{itemize}
    \item [(i)]  If $\beta=  \gamma = \theta \equiv 0$, the solution to the linear equation \eqref{eq : linear_Riccati_Y_common_noise} is $(Y,Z^Y) \equiv 0$. Moreover, by \eqref{eq : lambda_representation_common_noise} and assumption on model coefficients \eqref{ass:posR} and Remark \ref{remark:invO},  there exists  a positive constant $C_T' \geq 0$ such that for every $t \in [0,T]$
$- C_T' \leq \int_I \Lambda_t^u \d u \leq  C_T', $ $ \P \text{-a.s}$.
    \item [(ii)] Equations \eqref{eq : abstract Riccati barK} and \eqref{eq : linear_Riccati_Y_common_noise} are BSDEs on Hilbert spaces whereas \eqref{eq : standardRiccati K} and  \eqref{eq : RiccatiLambda} are standard BSDE respectively on $\R^{d \times d}$ and $\R$.
    \item [(iii)] The backward equation \eqref{eq : standardRiccati K} is independent of $\beta$,$\gamma$ and $\theta, G^A,G^C$, $G^Q$ and $G^H$ and the equation \eqref{eq : abstract Riccati barK} is independent of $\beta,\gamma$ and $\theta$.
    \item [(iv)] If all the model coefficients are set to be deterministic, the system of Equations \eqref{eq : standardRiccati K}, \eqref{eq : abstract Riccati barK}, \eqref{eq : linear_Riccati_Y_common_noise} and    \eqref{eq : RiccatiLambda} can be reduced to an ODE system of Riccati equations which will be clear from the proof of Proposition \ref{prop : fundamental_relation_common_noise}.
\end{itemize}

\end{Remark}

We prove a \emph{fundamental relation} which shows how to decompose the random cost functional $J(t,\boldsymbol{\xi},\balpha)$ defined in \eqref{eq:J_cost_functional} into a part independent with respect to  the control term $\alpha$ and a quadratic non-negative term.
\begin{Proposition}\textnormal{(Fundamental relation).}\label{prop : fundamental_relation_common_noise}
  Let $(K,Z^{K})$ be a solution of the standard Riccati equation \eqref{eq : standardRiccati K} on $[0,T]$, $(\bar{K},Z^{\bar{K}})$ be a solution to the abstract Riccati \eqref{eq : abstract Riccati barK} on $[0,T]$, $(Y,Z^{Y})$ be a  solution to  the linear BSRE equation \eqref{eq : linear_Riccati_Y_common_noise} on $[0,T]$ and $(\Lambda,Z^{\Lambda})$ a  solution to the linear equation \eqref{eq : RiccatiLambda} on $[0,T]$. Then, given $t \in [0,T]$ and $\boldsymbol{\xi} \in\Ic_t$, we have for every $\balpha \in \Ac$
\begin{align}\label{eq : fundamental_relation_common_noise}
         J(t,\boldsymbol{\xi},\balpha) &= \Vc(t,\boldsymbol{\xi}) +  \int_t^T  \int_I\E\Big[\Big\langle O^u_s\Big(\alpha^u_s + (O^u_s)^{-1}\big(U^u_sX^u_s + \int_I V_s(u,v)\bar X^v_s \d v + \Gamma_s^u \big)\Big),   \\
& \qquad \qquad  \qquad \qquad \quad \quad \alpha^u_s + (O^u_s)^{-1}\big(U^u_sX^u_s + \int_I V_s(u,v)\bar X^v_s \d v + \Gamma_s^u  \big)\Big\rangle \big | \tilde{\Fc}_t^0 \Big]\d u \d s \quad \P \text{-a.s},\notag
\end{align}
 where $\Vc$ is defined by
\begin{align} \label{eq : def_vxi}
    \mathcal{V}(t,\boldsymbol{\xi}) & := \int_I\E \big[\langle \xi^u, K^u_t \xi^u \rangle | \tilde{\Fc}_t^0 \big]\d u 
    + \int_I \int_I \langle\bar \xi^u,\bar K_t(u,v)\bar \xi^v \rangle \d v \d u  + \; 2 \int_I \E \big[\langle  Y^u_t,\xi^u\rangle | \tilde{\Fc}_t^0\big] \d u + \int_I \Lambda^u_t\d u,
\end{align} 
with $\bar{\xi}^u$ defined according to \eqref{eq : conditional_expectation_Xu} as $\bar{\xi}^u= \E[\xi^u| B^0_{. \wedge t}]$.

\end{Proposition}

\begin{proof}
    The proof is postponed to Appendix \ref{appendix : fundamental_relation}.
\end{proof}

\subsection{Solvability of the infinite dimensional Backward stochastic Riccati system}

In this subsection, we are going to show the solvability of the triangular Riccati BSDE systems \eqref{eq : standardRiccati K}, \eqref{eq : abstract Riccati barK},\eqref{eq : linear_Riccati_Y_common_noise} and \eqref{eq : RiccatiLambda} which arise from our optimal control problem.

\subsubsection{Solvability of $K$}\label{subsec : solvability_K}

As mentioned in $(iii)$ in Remark \ref{rmk : system_backward_equations}, Equation \eqref{eq : standardRiccati K} is independent of $\beta,\gamma,\theta$ and of $G^A, G^C,G^Q$ and $G^H$ so we can assume without loss of generality these coefficients to be null.

\begin{Assumption}\label{assumption : solva_K}
Assume $C_t^u(\omega)=0$, $D_t^u(\omega)=0, \quad \lambda(\d u ) \otimes \P(\d \omega)-\text{a.s.}$, for every $t \in [0,T]$.   
\end{Assumption}
This  assumption is used to ensure existence and uniqueness of a solution to the standard BSRE $(K,Z^K)$ which is essentially bounded for $\lambda(\d u )$-\text{a.e.} following standard theory. We refer to \cite{pham2016lqcmkv} (Section 3.2) for several other assumptions which are known in the literature to ensure existence and uniqueness of $(K,Z^K)$ under different sets of assumptions. In any case, under Assumption \ref{assumption : solva_K}  combined with Assumptions \ref{assumption : model coefficients} and \ref{assumption : cumulative_terminal_reward},  it is known from \cite{TangGeneral2003}  that $(K^u,Z^{K^u})$ admits a unique solution  with $K^u$ being symmetric positive and essentially bounded for $a.e.$ $u \in I$.
\begin{Proposition}\label{prop : Ricatti_K}
Under the Assumptions  \ref{assumption : solva_K}, there exists a unique solution $(K,Z^K)$ in the sense of Definition \ref{def : definition_solution_K} which satisfies
\begin{align}\label{eq : uniform_bound_K}
        \underset{t \in [0,T]}{\sup} |K_t^u(\omega) | \leq C , \quad \d u \otimes \d \P\text{ a.e}.
\end{align}
\end{Proposition}

\begin{proof}
\noindent \textit{Step n°1 : }
To ensure existence and uniqueness of $(K^u,Z^{K^u}) \in \S^2_{\tilde{\F}^0}(\R^{d \times d}) \times \H^2_{\tilde{\F}^0}(\R^{\d \times d})$ for $a.e.$ $u \in I$, we can rely on standard theory by  considering the following  standard linear-quadratic control problem with common noise,
 \begin{align}
    V^u(t,x)= \underset{\alpha \in \Ac}{\text{ ess inf }} \E \Big[  \int_{t}^{T} f^u(\tilde{X}_s, \alpha_s) + g^u(\tilde{X}_T) | \tilde{\Fc}_t^0 \Big]  ,
    \end{align}
    with dynamics
    $$\d \tilde{X}_s^u = b_t^u(\tilde{X}_s,\alpha_s) \d s + \sigma_t^u(\tilde{X}_s, \alpha_s) \d W_s^u,$$
    and with the $\tilde{\F}^0$-adapted random coefficients defined by
    $b_t^u(x,a) := A_t^u x + B_t^u a,$ $ 
    \sigma_t^u(x,a) :=  C_t^u x + D_t^u a,$ $
        f_t^u(x,a) := \langle x, Q_t^u x \rangle +  \langle a, R_t^u a \rangle, $ $
        g^u(x,a) := \langle x , H^u x \rangle.$ 
    Under Assumption \ref{assumption : solva_K}, we know that it is possible to show that $K^u \in \S^2_{\tilde{\F}^0}(\S^d_{+}) \cap L^{\infty}(\Omega;\Cc^{d \times d}_{[t,T]})$  where the essential supremum depends on the model parameters. Therefore, since the model coefficients are essentially bounded in $u \in I$, it follows that there exists a constant $C > 0$ such that
    \begin{align}\label{eq : uniform_bound_K}
        \underset{t \in [0,T]}{\sup} |K_t^u(\omega) | \leq C , \quad \d u \otimes \d \P\text{ a.e.}.
    \end{align}
    Moreover, following \eqref{eq : uniform_bound_K}, we get after integrating with respect to $I$ that $\int_I \E \Big[ \underset{0 \leq t \leq T}{\sup} |K_t^u|^2 + \int_{0}^{T} |Z_t^{K^u}|^2 \d t \Big] \d u < + \infty $.
    Therefore, we just need to verify that $K$ and $Z^K$ admit suitable representations in terms of Borel maps.

    \noindent \textit{Step n°2 : } We now show that our candidate $(K,Z^{K}) \in  \Sc^2_{\tilde{\F}^0}(I ; \R^{d \times d}) \times \Hc^2_{\tilde{\F}^0}(I; \R^{d \times d})$ (recall Section \ref{sec:notations}). For this, we will rely on a Picard scheme iteration by showing that the iterates satisfy the corresponding representation and taking the limit in the spirit of \cite{kharroubi2025stochastic}. For this, from a Picard scheme for BSDEs and using \eqref{eq : uniform_bound_K}, we have that 
     \begin{align}\label{eq : convergence_picard_scheme}
          \E \Big[ \int_I \Big[\underset{t \leq s \leq T}{\sup} | K_s^u - K_s^{u,n}|^2 + \int_{t}^{T} | Z_s^{K^u} - Z_s^{K^u,n}|^2  \d s \Big] \d u \Big] \underset{n \to \infty}{\to}  0.
     \end{align}
    where we define for $n \geq 1$ 
    \begin{align}\label{eq : Picard_iterations_K}
            \d K_t^{u,n} &= - \tilde{\Phi}^{u,C}( t,K_t^{u,n-1}) \d t + Z_t^{K^u,n} \d W_t^0,  \quad 
    K_T^{u,n} = H^u,
    \end{align}
    with $K_s^{u,0} = 0$ and $Z_s^{K^u,0} = 0$  for every $s \in [t,T]$ and where we introduce for   $C > 0$ (from \eqref{eq : uniform_bound_K})  the projection map $k \in \S^d \to \pi_C(k) \in \Bc(C)$ where $\Bc(C) := \lbrace k \in \S^d : | k | \leq C \rbrace$ as 
$\pi_C(k)=k$ if $|k| \leq C$, and $\pi_C(k)=C \frac{k}{|k|}$ if $|k|>C$
and $\tilde{\Phi}^{u,C}(t,k) = \tilde{\Phi}^{u}(t, \pi_C(k))$ which is Lipschitz uniformly in time.
    We now prove by induction on $n \in \N$ that there exist $\R^{d \times d}$ Borel measurable maps $k^n$ and $z^n$ defined on $I \times [0,T] \times \Cc_{[0,T]}$ such that
\begin{align}
        K_t^{u,n} = k^n(u,t,\tilde{B}^0_{. \wedge t}), \quad Z_t^{K^u,n} = z^n(u,t,\tilde{B}^0_{. \wedge t}), \quad \text{$\d t \otimes \d u \otimes \d \P$ a.e. },
    \end{align}
 This is clear for $n=0$. Suppose now that it holds true for $n-1$.  From \eqref{eq : Picard_iterations_K}, we get existence and uniqueness of $(K^{u,n}, Z^{K^{u,n})} \in \S^2_{\tilde{\F}^0}(\R^{d \times d})  \cap \H^2_{\tilde{\F}^0}(\R^{d \times d})$ from standard theory and we know that
    \begin{align}
        K_t^{u,n} &= \E \Big[ H^u + \int_{t}^{T} \tilde{\Phi}^{u,C}(s,K_s^{u,n-1}) \d s  | \tilde{\Fc}_t^0 \Big] = \E \Big[ h^T(u,\tilde{B}^0_{. \wedge T})  | \tilde{\Fc}_t^0 \Big] + \E \Big[ \int_{t}^{T}  \mathfrak{J}^{n-1}(u,s,\tilde{B}^0_{. \wedge s}) \d s  | \tilde{\Fc}_t^0 \Big]
    \end{align}
    where $ I \times [0,T] \times \Cc_{[0,T]} \ni (u,s,y) \mapsto \mathfrak{J}^{n-1}(u,s,y)  := \tilde{\Phi}^{u,C}(s, k^{n-1}(u,s,y))$ is a Borel measurable map due to the induction hypothesis  on $k^{n-1}$ and the assumption on the model coefficients implying the measurability of the random map $(u,t,k) \mapsto \tilde{\Phi}^{u,C}(t,k)$ (Recall that $\tilde{\Phi}^u$ has been defined in \eqref{eq : def_tilde_phi_u}.) Now, following the same proof as  in \cite[Section 5.3]{kharroubi2025stochastic}, we show the required form for $K^{u,n}$ and $Z^{K^u,n}$. From \eqref{eq : convergence_picard_scheme} which implies convergence in the space $L^{2}(I \times \Omega ; \Cc^{d \times d}_{[0,T]}) \times   L^2(I \times [0,T] \times \Omega ; \R^{d \times d})$, we complete the proof by setting along a subsequence  the functions
\begin{align}
        k(u,t,\tilde{B}^0_{. \wedge t})= \underset{n \to \infty}{\text{ lim }} k^n(u,t,\tilde{B}^0_{. \wedge t}), \quad z(u,t,\tilde{B}^0_{. \wedge t})= \underset{n \to \infty}{\text{lim  } } z^n(u,t,\tilde{B}^0_{. \wedge t}), \quad \d u \otimes \d t \otimes \d \P\text{ a.e.}.
    \end{align}
    Therefore, the proof is completed.

\end{proof}

\subsubsection{Solvability of $\bar{K}$}
As mentioned in $(iii)$ in Remark \ref{rmk : system_backward_equations}, equation \eqref{eq : abstract Riccati barK} is independent of $\beta,\gamma,\theta$. Therefore, without loss of generality, we assume in this subsection  $\beta = \gamma = \theta=0$.

\begin{Proposition}\textnormal{(A priori estimates).}\label{prop : a_priori_estimate_barKt}
    Let $T_0 \in [0,T]$ and $(\bar{K},Z^{\bar{K}})$ be the unique solution of the  abstract BSDE Riccati \eqref{eq : abstract Riccati barK} on $(T_0, T ]$. Then, there exists a positive constant $C_T >0$ such that 
\begin{align}\label{eq : bound_operator_barKt}
         \underset{t \in (T_0,T]}{\sup} \Vert T_{\bar{K}_t} \rVert_{\Lc(L^2(I;\R^d))} \leq C_T \quad \text{ $\P$-a.s.},
     \end{align} 
\end{Proposition}
\begin{proof}
     Let $t \in \hspace{0.05 cm} (T_0,T]$.  It follows from the fundamental relation \eqref{eq : fundamental_relation_common_noise} and from \eqref{eq : cost_functional_bound}.

     Recalling Lemma \ref{lemma : prop_randommap_F},  the operator $T_{\bar{K}_t}$ is $\P-\text{a.s.}$ self-adjoint. Indeed, we shall show that $\bar{K}_t = \bar{K}_t^{\star}$ in $\S^2_{\tilde{\F}^0}\big(L^2(I \times I ; \R^{d \times d}) \big)$ such that we have that $\P-\text{a.s.}$
     \begin{align}
         \langle z, T_{\bar{K}_t}(y) \rangle_{L^2(I;\R^d)} &= \int_I \int_I  z(u)^{\top} \bar{K}_t(u,v) y(v) \d v \d u , = \int_I \int_I z(v)^{\top} \bar{K}_t(v,u) y(u) \d v \d u , \\
         &=\int_I \int_I z(v)^{\top} \bar{K}^{\star}_t(u,v)^{\top} y(u) \d v \d u, = \int_I \int_I z(v)^{\top} \bar{K}_t(u,v)^{\top} y(u) \d v \d u , = \langle y, T_{\bar{K}_t}(z) \rangle_{L^2(I;\R^d)}.
     \end{align}
     Now, since $\bar{K}$ satisfies the BSRE \eqref{eq : abstract Riccati barK}, we have that $\bar{K}^{\star}$ satisfies over $(T_0,T]$
     \begin{align}
         \d \bar{K}_t^{\star} = -F(t, \bar{K}_t)^{\star} \d t + (Z_t^{\bar{K}})^{\star} \d W_t^0.
     \end{align}
     Hence, Lemma \ref{lemma : prop_randommap_F}  gives
     \begin{align}
     \begin{cases}
         d( \bar{K}_t - \bar{K}_{t}^{\star}) &= \big(Z_t^{\bar{K}} - (Z_t^{\bar{K}})^{\star}\big) \d W_t^0, \quad t \in (T_0,T] \\
         \bar{K}_T - \bar{K}_T^{\star} &=0,
     \end{cases}
     \end{align}
     where the terminal condition follows from $(G^H)^{\star} = G^H$. Hence, $\bar{K}_t - \bar{K}_t^{\star} = \E \big[ \int_{t}^{T} \big(Z_s^{\bar{K}} - (Z_s^{\bar{K}})^{\star} \big) \d W_s^0 | \tilde{\Fc}_t^0 \big] =0, \quad \P-\text{a.s}$.  Therefore, the operator $T_{\bar{K}_t}$ is $\P-\text{a.s.}$ self adjoint,
       and we recall that its operator norm is given  by  \eqref{eq: operator_norm_def}. Moreover, recalling that $\langle \bar{\xi}, T_{\bar{K}_t} \bar{\xi}  \rangle_{L^2(I;\R^d)} = \int_I \int_I \langle \bar{\xi}^u, \bar{K}_t(u,v) \bar{\xi}^v \rangle \d u \d v  $ and Remark \ref{rmk : system_backward_equations}, we have using the non-negativity of the cost functional $J$ stated in Remark \ref{rmk : conditional_functional_cost}
     \begin{align}\label{eq : upper_bound_operator_norm}
        \langle \bar{\xi}, T_{\bar{K}_t} \bar{\xi} \rangle_{L^2(I;\R^d)} \leq J(t,\boldsymbol{\xi},0) +C_T' \leq C_T \int_I \E \big[ |\xi^u|^2  | \tilde{\Fc}_t^0\big ] \d u + C_T', \quad \P-\text{a.s.} .
     \end{align} 
     From the fundamental relation \eqref{eq : fundamental_relation_common_noise} by taking the optimal control setting the quadratic term to be null and recalling Remark \ref{rmk : system_backward_equations}, we have
     \begin{align}\label{eq : lower_bound_operator_norm}
          \langle \bar{\xi}, T_{\bar{K}_t} \bar{\xi} \rangle_{L^2(I;\R^d)} \geq - C_T  \int_I \E \big[ |\xi^u|^2  | \tilde{\Fc}_t^0\big ] \d u - C_T', \quad \P-\text{a.s.}.
     \end{align}
     Therefore, taking deterministic initial random variables $\boldsymbol{\xi}=(\xi^u)_{ \in I}$ such that $\lVert \xi \rVert_{L^2(I;\R^d)}^2 = 1$, noticing that  $\bar{\xi} = \xi$, it follows from \eqref{eq : upper_bound_operator_norm} and \eqref{eq : lower_bound_operator_norm} that
  $  \underset{ \lVert \xi \rVert_{L^2(I;\R^d)}=1}{\sup}\big |  \langle \xi, T_{\bar{K}_t} \xi \rangle_{L^2(I;\R^d)} \big | \leq C_T $, $ \P-\text{a.s.},$
    which is enough to conclude by \eqref{eq: operator_norm_def}.
\end{proof}
\begin{Proposition}\label{prop : existence_unicity_barK}
 There exists a unique solution $(\bar{K}, Z^{\bar{K}})$ to \eqref{eq : abstract Riccati barK} in the sense of Definition \ref{def : eq_abstracticatti barK}.
\end{Proposition}
\begin{proof}
     We will prove local existence and uniqueness of  a solution on a small interval $[T- \delta, T]$ for a $\delta > 0$ to be chosen later. Then, using   Proposition \ref{prop : a_priori_estimate_barKt}, we will iterate the proof to construct a global solution over $[0,T]$ by showing that $\delta$ can be taken uniformly since it only depends on $\bar{K}_t$ through its operator norm $T_{\bar{K}_t}$  Similar arguments have been used in \cite{guatteri2006backward,de2026linear}. 
 To ease notations, we denote the Hilbert space $H:=L^2(I \times I ; \R^{d \times d})$ and we introduce for $\delta \in [0,T] $  $, r \in \R^+$ (to be chosen later)
        $\Bc(r,\delta):=\Big\{\bar{K} \in S^2_{\tilde{\F}^0}(H), \ \sup_{t\in[T-\delta,T]}\|T_{\bar K_t}\|_{\Lc(L^2(I;\R^d))}\le r, \quad \P \text{-a.s}  \Big \rbrace,$
    where we use  the  norm $\lVert \bar{K} \rVert^2_{\Sc^2(H)} :=\E \Big[ \underset{ t \in [T- \delta, T]}{\sup}  \lVert \bar{K}_t \rVert^2_{H}\Big]$.
    \noindent On $\Bc(r,\delta)$, we  construct the map
\begin{align}
   \Phi : \Bc(r,\delta) \to \Bc(r,\delta),\quad  \Phi(\bar{K}) = \bar{P},\quad  \textit{where }  \bar{P}_t :=   \E \Big[ G^H + \int_t^T F(s, \bar{K}_s) \d s  |\tilde{\Fc}_t^0 \Big].
    \end{align}
    We now prove that $\Phi$ is a well-defined map and a contraction on $\Bc(r,\delta)$ for suitable choices of $\delta$ and $r$.

    \noindent \textbf{Step n°1 : $\Phi$ is a well-defined map. } We  need to prove  that $ \underset{t \in [T- \delta , T ]}{\sup}  \lVert T_{\bar{P}_t} \lVert_{\Lc(L^2(I;\R^d)} \leq r \quad \text{$\P$ a.s}.$ since it is clear that $\bar{P}\in \S^2_{\tilde{\F}^0}(H)$. Using  Jensen's inequality and  Lemma \ref{lemma : prop_randommap_F} (1), we have
\begin{align}\label{eq : bound_operatornorm_Pt}
\lVert T_{\bar{P}_t} \rVert_{\Lc(L^2(I;\R^d))} &\leq  \E \Big[ \lVert T_{G^H} \rVert_{\Lc(L^2(I;\R^d))} + \int_{t}^{T}  \lVert T_{F(s,\bar{K}_s)} \rVert_{\Lc(L^2(I;\R^d))} \d s | \tilde{\Fc}_t^0 \Big] \\
& \leq \E \Big[ \lVert T_{G^H} \rVert_{\Lc(L^2(I;\R^d))} + \int_{t}^{T}  C_T^1(1+ \lVert T_{\bar{K}_s} \rVert_{\Lc(L^2(I;\R^d))} + \lVert T_{\bar{K}_s} \rVert_{\Lc(L^2(I;\R^d))}^2 ) \d s |\tilde{\Fc}_t^0 \Big ] \\
&\leq M^H + \delta  C_T^1(1+ r + r^2) , \quad \P \text{-a.s},
\end{align}
where $M^H$ denotes the essential supremum over $\Omega$ of $\lVert T_{G^H} \rVert$. By Proposition \ref{prop : a_priori_estimate_barKt}, we have $M^H \leq C_T$
Taking the supremum on $[T- \delta, T]$ on \eqref{eq : bound_operatornorm_Pt} gives
    $\underset{t \in [T- \delta, T]}{\sup} \lVert T_{\bar{P}_t} \rVert_{\Lc(L^2(I;\R^d))}  \leq M^H + \delta  C_T^1(1+ r + r^2), $ $ \P \text{-a.s}.$
Taking $r > C_T \geq M_T $ and $\delta$  such that $C_T + \delta  C_T^1(1+ r + r^2) \leq r$ is enough to ensure that $\Phi( \Bc(r,\delta)) \subset \Bc(r,\delta)$.

    \noindent \textbf{Step n°2 : $\Phi$ is a contraction. }
    Let $\bar{K}_1,\bar{K}_2 \in  \Bc(r,\delta)$. We denote $\Delta \bar{K}_t := \bar{K}_t^1 - \bar{K}_t^2$ and $\Delta \bar{P}_t = \bar{P}_t^1 - \bar{P}_t^2$. 
     By definition and conditional Jensen's inequality, we have
     \begin{align}\lVert \Delta \bar{P}_t \rVert_{\Sc^2(H)}=\left\lVert\E \Big[  \int_{t}^{T} F(s,\bar{K}_s^1 )- F(s,\bar{K}_s^2) \d s| \tilde{\Fc}_t^0 \Big]  \right\rVert_{\Sc^2(H)} \leq \mathbb E \left[ \underset{t \in [T- \delta, T]}{\sup}\int_{t}^{T} \E \Big[ \lVert F(s,\bar{K}_s^1)- F(s,\bar{K}_s^2) \rVert_{H} |\tilde{\Fc}_t^0 \Big ] \d s \right]\end{align}
    By  Lemma \ref{lemma : prop_randommap_F} and since $\bar{K}^1, \bar{K}^2 \in \Bc(r)$, we have 
    \begin{align}
           \lVert F(s,\bar{K}_s^1)- F(s,\bar{K}_s^2) \rVert_{H} &\leq C\big( 1+  \lVert T_{\bar{K}^1_s} \rVert_{\Lc(L^2(I;\R^d))} +  \lVert T_{\bar{K}^2_s} \rVert_{\Lc(L^2(I;\R^d))} \big)  \lVert \bar{K}_s^1 - \bar{K}_s^2 \rVert_{H}, \leq C(1+2r)  \lVert \Delta \bar{K}_s \rVert_{H}, \quad \P \text{-a.s}.
    \end{align}
    It implies that
      \begin{align}
        \lVert \Delta \bar{P} \rVert_{\Sc^2(H)} \leq  C(1+2r) \delta  \lVert \Delta \bar{K} \rVert_{\Sc^2(H)}.
    \end{align}
    Therefore, taking $\delta$ such that $C(1+2r) \delta < 1$ together with the previous condition implies that $\Phi$ is a contraction on the complete metric space $(\Bc(r), \lVert \cdot \Vert)$ as a closed subset of the Banach space $(\S^2_{\tilde{\F}^0}(H); \lVert \cdot \rVert)$ which ensures therefore existence and uniqueness of a fixed-point  $\bar{K} \in \Bc(r)$ which satisfies
    \begin{align}
        \bar{K}_t = \E \Big[ G_H  + \int_{t}^{T} F(s,\bar{K}_s) \d s | \tilde{\Fc}_t^0 \Big], \text{ on $[T-\delta, T]$}.
    \end{align}
    Then, by \cite[Proposition 6.18]{fabbri2017stochastic}, the process $Z^{\bar{K}}$ is uniquely defined in $\H^2_{\tilde{\F}^0}(H)$ as the martingale decomposition applied to the martingale defined by
    \begin{align}
        M_t = \E \Big[ \bar{K}_t + \int_{0}^{t} F(s,\bar{K}_s) \d s  | \tilde{\Fc}_t^0 \Big].
    \end{align}
    Therefore, we have existence and uniqueness of $(\bar{K},Z^{\bar{K}})$ on $[T- \delta, T]$.

    Now, observing that $\delta$ only depends on $\bar{K}_t$ through $\lVert T_{\bar{K}_t} \rVert_{\Lc(L^2(I;\R^d))}$ (see \eqref{eq : bound_operatornorm_Pt}) for which we have the a priori estimate \eqref{eq : bound_operator_barKt}, we can iterate the proof on every interval of length $\delta$ given $r  > C_T$ from  \eqref{eq : bound_operator_barKt}  which provides global existence and uniqueness for $(\bar{K},Z^{\bar{K}})\in \Sc^2_{\tilde{\F}^0}(H) \times \H^2_{\tilde{\F}^0}(H)$ on $[0,T]$. 
    Moreover, observe that  $\P -a.s$, continuity of the map $[t,T] \ni s  \mapsto T_{\bar{K}_s}(\omega) \in \Lc(L^2(I;\R^d)) $ is obtained recalling \eqref{eq : operator_norm_inequality}  since $\bar{K}$ admits a continuous version.
\end{proof}
\subsubsection{Solvability of $Y$}
\begin{Proposition}
    There exists a unique solution $(Y,Z^{Y})$ of equation \eqref{eq : linear_Riccati_Y_common_noise} in the sense of Definition \ref{def : linear_Riccati_Y}.
\end{Proposition}
\begin{proof}
    Following Remark \ref{rmk : terms_Y} and the previous results on $K$ and $\bar{K}$, we have
    \begin{enumerate}
        \item [$\bullet$] For any $t \in [0,T]$, we have $\lVert A_t \Vert_{\Lc(L^2(I;\R^d))} \leq C, \quad \P \text{-a.s}$
        for a uniform constant $C \geq 0$ which ensures the uniform-in-time Lipschitz property for the driver $\hat{F}$ in \eqref{eq : driver_tildeF_second_form}.
        \item [$\bullet$] $\E \Big[ \int_{0}^{T} \lVert f_t \rVert_{L^2(I;\R^d)}^2  \d t \Big] < + \infty$ and the integrability of the terminal condition.
  
    \end{enumerate}
Hence, the result follows from \cite[Proposition 6.20]{fabbri2017stochastic}.
\end{proof}
\subsection{Verification theorem}

The  fundamental relation, equation \eqref{eq : fundamental_relation_common_noise},  suggests us to consider the optimal control defined for  every $s \in [t,T]$ and for a.e. $u \in I$ in the following feedback form 
\begin{align}\label{eq :feedback_control_form_thm_common_noise}
    \alpha_s^u  = - (O_s^u)^{-1} \Big(U_s^u X_s^u + \int_I V_s(u,v)   \bar{X}_s^v  \d v + \Gamma_s^u  \Big), \quad  \P-\text{a.s.},  
\end{align}
where we consider the closed-loop SDE dynamics for $\bX=(X^u)_{u \in I}$ as
\begin{align}\label{eq : feedback_closed_loop_SDE_common_noise}
    dX_s^u &=  \Big[ \beta_s^u - B_s^u (O_s^u)^{-1} \Gamma_s^u  + \big(A_s^u - B_s^u (O_s^u)^{-1} U_s^u \big) X_s^u   + \int_I \big(G_s^A(u,v) - B_s^u (O_s^u)^{-1} V_s(u,v) \big) \bar{X}_s^v  \d v  \Big] \d s , \\
    &\quad+ \Big[  \gamma_s^u - B_s^u (O_s^u)^{-1} \Gamma_s^u  + \big(C_s^u - B_s^u (O_s^u)^{-1} U_s^u \big) X_s^u   + \int_I \big(G_s^C(u,v) - B_s^u (O_s^u)^{-1} V_s(u,v) \big) \bar{X}_s^v  \d v  \Big  ] \d W_s^u, \\
    & \quad +  \theta_s^u  \d \tilde{B}_s^0,
\end{align}
and 
$ X_t^u = \xi^u, u \in I.$ Since \eqref{eq : feedback_closed_loop_SDE_common_noise} is a linear mean-field SDE with common noise with coefficients satisfying Assumption \ref{assumption : model coefficients}, from Theorem \ref{thm: existence_unicityX} we have existence and uniqueness for $\bX$ and therefore admissibility for the feedback control \eqref{eq :feedback_control_form_thm_common_noise}. Indeed, the integrability condition follows from \eqref{eq : uniform_bound_K}, \eqref{prop : a_priori_estimate_barKt}, the assumption on the model coefficients and the existence and uniqueness of $\bX$ while the representation of $\alpha$ through a suitable Borel map follows from  the assumption on the model coefficients and by the representations of the solutions of $K,\bar{K}, Y, \Lambda$ (recall definitions of solutions). 
\begin{Theorem}\label{thm : optimal_control_form}
   Let $(K^u,Z^{K^u})$ be the unique solution of the standard Riccati BSDE equation \eqref{eq : standardRiccati K} on $[0,T]$, $(\bar{K},Z^{\bar{K}})$ be the unique solution of the abstract Riccati BSDE equation \eqref{eq : abstract Riccati barK} on $[0,T]$, $(Y,Z^{Y})$ be the unique solution to  \eqref{eq : linear_Riccati_Y_common_noise} on $[0,T]$ and $\Lambda$ be the unique solution to \eqref{eq : RiccatiLambda} on $[0,T]$. Then:
    \begin{enumerate}
        \item[(1)] The unique optimal control $\hat{\balpha}$ is in feedback form and is given for $a.e.$ $u \in I$ by \eqref{eq :feedback_control_form_thm_common_noise},  where $(X_s^u)_{t \leq s \leq T}$ is the unique solution to the closed-loop equation \eqref{eq : feedback_closed_loop_SDE_common_noise}
        \item [(2)] The value function $V(t,\boldsymbol{\xi}) := \underset{\balpha \in \Ac}{\text{ ess inf  }} J(t,\boldsymbol{\xi},\balpha)$ is given by 
        \begin{align}
            V(t,\boldsymbol{\xi} ) = \int_I \E \Big[\langle \xi^u, K^u_t \xi^u \rangle | \tilde{\Fc}_t^0 \Big]\d u 
    + \int_I \int_I  \E \Big[\langle\bar \xi^u,\bar K_t(u,v)\bar{\xi}^v \rangle \Big] \d v \d u  + \; 2 \int_I \E \Big[\langle  Y^u_t,\bar{\xi}^u \rangle \Big] \d u + \int_I \Lambda^u_t\d u, \quad \P \text{-a.s}.
        \end{align} 
    \end{enumerate}
\end{Theorem}
\begin{proof}
    From the fundamental relation \eqref{eq : fundamental_relation_common_noise}, it follows that
    \begin{enumerate}
        \item [$\bullet$] $J(t,\boldsymbol{\xi},\balpha) \geq \Vc(t,\boldsymbol{\xi})$ for any $\balpha \in \Ac$ where $\Vc$ is defined in \eqref{eq : def_vxi}.
        \item [$\bullet$] We have equality if $\balpha$ satisfies \eqref{eq :feedback_control_form_thm_common_noise} in order to set  the quadratic term in \eqref{eq : fundamental_relation_common_noise} equal to 0. Moreover, we notice that $\hat{\balpha} \in \Ac$ following Theorem \ref{thm: existence_unicityX} and existence and uniqueness of \eqref{eq : feedback_closed_loop_SDE_common_noise} ensure the optimality of a unique control given by $\hat{\balpha}$ defined in \eqref{eq :feedback_control_form_thm_common_noise} .
    \end{enumerate}
\end{proof}

\normalsize

\section{Applications}\label{sec:applications}
In this section, we apply the theory developed to solve two problems arising in mathematical finance.  
\subsection{Optimal trading with heterogeneous market participants}
We introduce an optimal trading problem under heterogeneous market participants in a cooperative setting with common noise (see the introduction for references on the subject).

In particular, we consider a large bank with traders indexed by $u \in I$ who act under the guidance of a lead trader. We assume that the inventory of each trader $u \in I$ is given by the following SDE
\begin{align}
        \d X_t^u =   \alpha_t^u   \d t + \sigma_t^u \d W_t^u + \sigma_t^{0,u}   \d \tilde{B}_t^0 ,  \quad
    X_0^u = \xi^u,
\end{align}
where $\boldsymbol{\xi}:=(\xi^u)_u$ is an admissible initial condition in the sense of Definition \ref{def : admissible_conditions} and we aim to minimize the following cost functional  given by 
\begin{align}
    J(t,\boldsymbol{\xi},\balpha)=   \int_I \E \Big[ \int_{t}^{T} (\alpha_s^u + P^u)^2  \d s + \lambda^u \langle X_T^u - \int_I \tilde{G}^{\lambda}(u,v) \bar{X}_T^v \d v  , X_T^u - \int_I \tilde{G}^{\lambda}(u,v) \bar{X}_T^v  \d v  \rangle | \tilde{\Fc}_t^0  \Big] \d u, 
\end{align}
where $P^u \geq 0$ and $\lambda^u \geq 0$ denote, respectively the transaction price and risk aversion parameter for agent $u$, and where $\tilde{G}^{\lambda} \in L^2(I \times I;\R)$ is a symmetric graphon modeling the interactions between traders $u$ and $v$. Moreover, $\sigma^u$ and $\sigma^{0,u}$ are $\tilde{\F}^0$-adapted processes that denote, respectively the idiosyncratic noise and a common noise on the inventory of the agent $u \in I$.
Following Remark \ref{rmk : alternative_formulation_centered}, we have 
\begin{align}
    J(t,\boldsymbol{\xi},\balpha) =  \int_I \E \Big[ \int_{t}^{T} (\alpha_s^u + P^u)^2  \d s + \langle X_T^u,\lambda^u X_T^u \rangle +  ( \bar{X}_T^u ,\int_I G^{\lambda}(u,v) \bar{X}_T^v)  |\tilde{\Fc}_t^0\Big] \d u, 
\end{align}
where we denote $G^{\lambda}(u,v) :=  \big(\tilde{G}^{\lambda}  \circ \lambda \circ \tilde{G}^{\lambda} \big)(u,v) - (\lambda^u + \lambda^v) \tilde{G}^{\lambda}(u,v)$. 

This models falls into the framework  of the non-exchangeable L-Q mean-field control problem with common noise \eqref{eq : state equations LQ common noise} - \eqref{eq:J_cost_functional} with $d=m=1$ and coefficients given for $t \in [0,T]$ and $u,v \in I$ by $\beta_t^u = 0,$ $A_t^u = 0,$ $G_t^A(u,v) = 0, $ $ B_t^u = 1, $ $\gamma_t^u = \gamma^u,$ $C_t^u = 0,  $ $ G_t^C(u,v) = 0, $ $D_t^u = 0,$  $\theta_t^u = \sigma^u,$ $ Q_t^u = 0,$ $G_t^Q(u,v) = 0,$ $R_t^u = 1,$ $I_t^u = P^u,$ $H^u = \lambda^u, $ $G^H(u,v)= G^{\lambda}(u,v).$

Following the notations introduced in \eqref{eq : notations_common_noise_Kt_no_control_vol}, \eqref{eq : notations_barKt_common_noise_no_control_vol}, \eqref{eq : notations_common_noise_Yt_no_control_vol} and \eqref{eq : coeff_Lambda}, we have
$\Phi_t^u = 0,$ $U_t^u= K_t^u, $ $O_t^u = 1,$ $\Psi_t(u,v) = 0,$ $V_t(u,v) = \bar{K}_t(u,v),$ $M_t(u)= Z_t^{K^u} \sigma^{0,u}  + \int_I Z_t^{\bar{K}}(u,v) \sigma^{0,v}  \d v - \int_I \kappa \tilde{G}^{\kappa}(v,u) Y_t(v) \d v $ $\Gamma_t^u = Y_t(u) + P^u,$ $
L_t^u = K_t^u \big( (\sigma^u)^2 + (\sigma^{0,u})^2\big) + \sigma^{0,u} Z_t^{Y}(u) + \int_I \sigma^{0,u} \sigma^{0,v} \bar{K}_t(u,v) \d v +(P^u)^2 - (Y_t(u) + P^u)^2.$
We assume that  Assumption \ref{assumption : model coefficients} is satisfied. Moreover, we notice that  Assumption \ref{assumption : solva_K} is directly satisfied. Therefore,  Theorem \eqref{thm : optimal_control_form} applies, that is, we obtain the optimal  value function  in the following quadratic form
    \begin{align} \label{eq : def_vxi optimal_trading}
    \Vc(t,\boldsymbol{\xi}) = \int_ I \E[ K_t^u (\xi^u)^2 |\tilde{\Fc}_t^0 ]\d u 
    + \int_ I\int_ I \bar \xi^u\bar K_t(u,v)\bar \xi^v \d v \d u \ + \; 2 \int_I\E[  Y^u_t\bar{\xi}^u] \d u + \int_U \Lambda^u_t\d u.
\end{align} 
where for  a.e. $u \in I$, the Backward Riccati equation for $K_t^u $ is given by 
\begin{align}
    d K_t^u = - \Big[ - (K_t^u)^2\Big] \d t + Z_t^{K^u} \d W_t^0, \quad t \in [0,T), \quad 
    K_T^u = \lambda^u.  
\end{align}
The Abstract Riccati BSDE on the Hilbert space $L^2(I \times I ;\R)$ for $\bar{K}$ is given by 
\begin{align}
    \d \bar{K}_t = - F(t,\bar{K}_t)\d t+ Z_t^{\bar{K}} \d W_t^0, \quad 
    \bar{K}_T = G^{\lambda},
\end{align}
where we have 
    $F(t,\bar{K_t})(u,v) = - (\bar{K}_t^* \circ \bar{K}_t)(u,v).$
The Abstract linear BSDE on the Hilbert space $L^2(I;\R)$ for $Y$ is given by 
\begin{align}
    Y_t = - \tilde{F}(t,Y_t)\d t+ Z_t^{Y} \d W_t^0, \quad 
    Y_T= 0,   
\end{align}
where we have
    $\tilde{F}(t,Y_t)(u) = M_t(u) - \int_I V_t(v,u) \Gamma_t^v \d v. $
Following \eqref{eq : lambda_representation_common_noise}, $\Lambda^u$ is given by
\begin{align}
    \Lambda_t^u =  \int_{t}^{T} \E \Big[  L_s^u  | \tilde{\Fc}_t^0 \Big] \d s, \quad t \in [0,T], \quad \text{a .e $u \in I$.}
\end{align}
Moreover, the optimal control is given according to Theorem \ref{thm : optimal_control_form} in feedback form by
\begin{align}
    \hat{\alpha}_s^u = - K_s^u X_s^u - \int_I \bar{K}_s(u,v) \bar{X}_s^v \d v - Y_s(u) - P_s^u .
\end{align}

\begin{Remark}
 Setting all the coefficients to be independent of $u \in I$, the graphon term $\tilde{G}^{\lambda}$ to be equal to 1 and assuming $\sigma^{0,u} = 0$ which implies no common noise, then we again recover the result in \cite{frikha2025actorcritic} for the derivation of optimal control after noticing that $Y=0$ and $K \equiv -\bar{K}$ in this specific setting.
\end{Remark}

\subsection{Systemic risk with heterogeneous banks}
Here, we introduce and solve a systemic risk model with heterogeneous banks under common noise (see the introduction for references on the subject). We consider a model of interbank borrowing and lending of a continuum of heterogeneous banks, with the log-monetary reserve of each bank $u \in I$ given by
\begin{align}
        \d X_t^u =  \Big[ \kappa (X_t^u - \int_I \tilde{G}^{\kappa} (u,v) \bar{X}_t^v   \d v ) + \alpha_t^u   \Big] \d t + \sigma_t^u \Big(  \sqrt{1- (p_t^u)^2}  \d W_t^u + \rho_t^u  \d \tilde{B}_t^0 \Big ),  \quad
    X_0^u = \xi^u.
\end{align}
where $\boldsymbol{\xi}:=(\xi^u)_u$ is an admissible initial condition in the sense of Definition \ref{def : admissible_conditions}, where $\kappa \leq 0$ is the rate of mean-reversion in the interaction from borrowing and lending between the banks, $\sigma^u$ is the $\tilde{\F}^0$-adapted volatility process of the $u$-labeled bank reserve and  $\tilde{G}^{\kappa}\in L^2(I \times I;\R)$ is a symmetric graphon modeling the interactions between banks $u$ and $v$. Here, there is also a common noise $\tilde{B}^0$ which affects bank $u$ through the term $\sigma^u \rho^u$, where $\rho^u$ is an $\tilde{\F}^0$-adapted process valued in $[-1,1]$ acting as a correlation parameter.
The aim of the central bank is now to mitigate systemic risk by minimizing the following aggregate conditional cost functional
\begin{align}
      \int_I \E \Big[ \int_{t}^{T} \Big( \eta^u \big(X_s^u - \int_I \tilde{G}^{\eta}(u,v) \E [ X^v_s |\tilde{\Fc}_s^0 ] \d v \big)^2  + (\alpha_s^u)^2\Big) \d s+ r^u \big( X_T^u  - \int_I  \tilde{G}^r(u,v) \E [ X^v_T |\tilde{\Fc}_T^0 ] \d v  \big)^2| \tilde{\Fc}_t^0 \Big] \d u.
\end{align}
where $\eta^u,r^u>0$ penalize deviations from the  weighted average, where $(\alpha_s^u)^2$ represents the cost of borrowing/lending from the central bank. By Remark \ref{rmk : alternative_formulation_centered}, we have
\begin{align}
    J(t,\boldsymbol{\xi},\balpha) := \int_I \E \Big[ \int_{t}^{T} \Big( \eta^u (X_s^u)^2 + \bar{X}_s^u \int_I G^{\eta}(u,v) \bar{X}_s^v \d v   + (\alpha_s^u)^2  \Big) \d s+ r^u (X_T^u)^2 +  \bar{X}_T^u \int_I  G^r(u,v) \bar{X}_T^v  \d v  | \tilde{\Fc}_t^0 \Big] \d u,
\end{align}
where we denote $G^{\eta}(u,v) :=  \big(\tilde{G}^{\eta}  \circ \eta \circ \tilde{G}^{\eta} \big)(u,v) - (\eta^u + \eta^v) \tilde{G}^{\eta}(u,v)$ and $G^{r}(u,v) :=  \big(\tilde{G}^{r}  \circ r \circ \tilde{G}^{r} \big)(u,v) - (r^u + r^v) \tilde{G}^{r}(u,v)$.
This model falls into the framework  of the non-exchangeable L-Q mean-field control problem with common noise \eqref{eq : state equations LQ common noise} - \eqref{eq:J_cost_functional}  with $d=m=1$ and coefficients given for $t \in [0,T]$ and $u,v \in I$ by
$\beta_t^u = 0,$ $A_t^u= \kappa$, $G_t^A(u,v)= -\kappa,$ $\tilde{G}^{\kappa}(u,v)$, $B_t^u = 1,$ $\gamma_t^u = \sigma^u \sqrt{1 - (\rho^u)^2}, $ $
C_t^u = 0,$ $G_t^C(u,v) = 0$ $D_t^u = 0,$ $\theta_t^u = \sigma^u \rho^u,$ $Q_t^u = \eta^u,$ $ G_t^Q(u,v) = G^{\eta}(u,v),$ $R_t^u = 1,$ $I_t^u=0,$ $H^u = r^u,$ $G^H(u,v) = G^{r}(u,v).$
Following the notations introduced in \eqref{eq : notations_common_noise_Kt_no_control_vol}, \eqref{eq : notations_barKt_common_noise_no_control_vol}, \eqref{eq : notations_common_noise_Yt_no_control_vol} and \eqref{eq : coeff_Lambda}, we have $\Phi_t^u = 2\kappa$  $K_t^u + \eta^u,$ $U_t^u = K_t^u,$ $O_t^u = 1,$ $\Psi_t(u,v) = -\kappa,$ $ K_t^u\, \tilde{G}_{\kappa}(u,v)
              - \kappa\, \tilde{G}_{\kappa}(v,u)\, K_t^v
              + 2\kappa\, \bar{K}_t(u,v) - \int_I \kappa\, \tilde{G}_{\kappa}(w,u)\, \bar{K}_t(w,v) \mathrm{d}w- \int_I \kappa\,\tilde{G}_{\kappa}(w,v)\, \bar{K}_t(u,w)\, \mathrm{d}w
        + G_{\eta}(u,v),$ 
$V_t(u,v) = \bar{K}_t(u,v),$ $M_t(u)= Z_t^{K^u} \sigma^u \rho^u + \int_I Z_t^{\bar{K}}(u,v) \sigma^v \rho^v \d v - \int_I \kappa \tilde{G}^{\kappa}(v,u) Y_t(v) \d v $ $\Gamma_t^u = Y_t(u),$ $L_t^u = (\sigma^u)^2 K_t^u + \sigma^u \rho^u Z_t^{Y}(u) + \int_I \sigma^u \sigma^v \rho^u \rho^v \bar{K}_t(u,v) \d v - (Y_t(u))^2$. We assume that  Assumption \ref{assumption : model coefficients} is satisfied. Moreover, we notice that  Assumption \ref{assumption : solva_K} is directly satisfied. Therefore,  Theorem \ref{thm : optimal_control_form} applies, that is, we obtain the optimal  value function  in the following quadratic form
    \begin{align} \label{eq : def_vxi systemic risk}
    \Vc(t,\boldsymbol{\xi}) = \int_ I \E[ K_t^u (\xi^u)^2 |\tilde{\Fc}_t^0 ]\d u 
    + \int_ I\int_ I \bar \xi^u\bar K_t(u,v)\bar \xi^v \d v \d u \ + \; 2 \int_I\E[  Y^u_t\bar{\xi}^u] \d u + \int_U \Lambda^u_t\d u.
\end{align} 
where for  a.e. $u \in I$, the Backward Riccati equation for $K_t^u $ is given by 
\begin{align}
    d K_t^u = - \Big[ 2 \kappa K_t^u + \eta^u - (K_t^u)^2\Big] \d t + Z_t^{K^u} \d W_t^0, \quad t \in [0,T) \quad 
    K_T^u = r^u.  
\end{align}
The Abstract Riccati BSDE on the Hilbert space $L^2(I \times I ;\R)$ for $\bar{K}$ is given by 
\begin{align}
    \d \bar{K}_t = - F(t,\bar{K}_t)\d t+ Z_t^{\bar{K}} \d W_t^0, \quad 
    \bar{K}_T = G^r,
\end{align}
where 
\begin{align}
    F(t,\bar{K_t})(u,v) &= - \kappa K_t^u \tilde{G}^{\kappa}(u,v) - \kappa \tilde{G}^{\kappa}(v,u) K_t^v + 2 \kappa \bar{K}_t(u,v) - \kappa (\tilde{G}^{\kappa} \circ \bar{K}_t)(u,v) - \kappa({\bar{K}}_t^* \circ \tilde{G}^{\kappa})(u,v)  \\
    &\quad + G^{\eta}(u,v) - (\bar{K}_t^* \circ \bar{K}_t)(u,v).
\end{align}
The Abstract linear BSDE on the Hilbert space $L^2(I;\R)$ for $Y$ is given by 
\begin{align}
    Y_t = - \tilde{F}(t,Y_t)\d t+ Z_t^{Y} \d W_t^0, \quad 
    Y_T= 0, 
\end{align}
where we have
\begin{align}
    \tilde{F}(t,Y_t)(u) = Z_t^{K^u} \sigma^u \rho^u + \int_I Z_t^{\bar{K}}(u,v) \sigma^v \rho^v \d v - \int_I \kappa \tilde{G}^{\kappa}(v,u) Y_t(v) \d v -  \int_I V_t(v,u) Y_t(v) \d v.
\end{align} 
Following \eqref{eq : lambda_representation_common_noise}, $\Lambda^u$ is given by
\begin{align}
    \Lambda_t^u =  \int_{t}^{T} \E \Big[ (\sigma^u)^2 K_s^u + \sigma^u \rho^u Z_s^Y(u) + \int_I \sigma^u \sigma^v \rho^u \rho^v  \bar{K}_s(u,v) \d v  | \tilde{\Fc}_t^0 \Big] \d s, \quad t \in [0,T], \quad \text{a .e $u \in I$.}
\end{align}
Moreover, the optimal control is given according to Theorem \ref{thm : optimal_control_form} in feedback form as
\begin{align}
    \hat{\alpha}_s^u = - K_s^u  X_s^u    - \int_I   \bar{K}_s(u,v)  \bar{X}_s^v \d v   - Y_s(u).
\end{align}
\begin{Remark}   Setting all the coefficients to be independent of $u \in I$, and the graphon terms $\tilde{G}^{\kappa}, \tilde{G}^{\eta}$ and $\tilde{G}^{r}$ to be equal to 1, we recover after some straightforward computations the result in \cite{pham2017dynamic} (Section 5).
\end{Remark}
\appendix
\section{Proof of Theorem \ref{thm: existence_unicityX}}\label{appendix: proof_existence_unicityX}

    Fix $t \in [0,T]$, $\boldsymbol{\xi} \in \Ic_t$ and $\boldsymbol{\alpha} \in \Ac$. We define the map $\Phi$ on the space $\Sc^2_{\tilde{\F}^0}(I;\R^d)$ (which is Banach by standard arguments) as follows. For every $\bY:=(Y^u)_{u \in I} \in \Sc^2_{\tilde{\F}^0}(I;\R^d)$, we define
    \begin{align}
        \Phi(\bY) = \bar{X}^{\bY} =(\bar{X}^{\bY,u})_{u \in I}, \quad \bar{X}_s^{\bY,u} := \E[X_s^{\bY,u} | \tilde{B}^0_{. \wedge s}], \quad s \in [t,T].
    \end{align}
where for every $u \in I$, $X^{\bY,u}$ is the solution to the standard SDE with common noise
    \begin{align}\label{eq : state equations LQ common noise proof}
\begin{cases}
    &\d X^{\bY,u}_s =\left [ \beta_s^u + A_s^u X_s^{\bY,u} +  \int_I G^A_s(u,v) Y_s^v  \d v + B_s^u \alpha^u_s \right] \d s \\
    &\qquad\qquad +\left [\gamma_s^u + C_s^u  X^{\bY,u}_s +    \int_I G^C_s(u,v)  Y_s^v\d v + D^u_s \alpha^u_s \right ] \d W^u_s +\theta^u_s  \d \tilde{B}_s^0,\quad s\in[t,T]\\ 
    & X_t^u=\xi^u, \qquad u \in U, 
\end{cases}
\end{align}
\paragraph{Step 1 :}
We first show that $\Phi$ is a well-defined map.
Assuming for now that $u \to \P_{(X^{\bY,u},W^u, \tilde{B}^0,R)} \in \Pc( \Cc^d_{[t,T]} \times \Cc_{[0,T]} \times \Cc_{[0,T]} \times \mathcal{K}) $ is measurable, then it implies the measurability of the map $u \mapsto \P_{(X^{\bY,u},W^u, \tilde{B}^0,R, \alpha^u)}$ and from standard estimates on Equation \eqref{eq : state equations LQ common noise}, we have
\begin{align}\label{eq : integrability_condition_X_appendix}
    \E \Big[ \underset{s \in [t,T]}{\sup} |X_s^{\bY,u}|^2 \Big] \leq C_T \Big( 1 + \E \big[ | \xi^u|^2 \big] + \int_{t}^{T} \E \big[ |\alpha_s^u|^2 \big ] \d s + \lVert \bY \rVert_{\Sc^2_{\tilde{\F}^0}(I;\R^d)}^2 \Big)< \infty,
\end{align}
where we have used  \eqref{eq : hypothesis initial conditions} and \eqref{eq : hypothesis admissible control}. By  Jensen's inequality,  $\int_I\E \Big[ \underset{s \in [t,T]}{\sup} |\bar{X}_s^{\bY,u}|^2 \Big] du \leq \int_I \E \Big[ \underset{s \in [t,T]}{\sup} |X_s^{\bY,u}|^2 \Big]du< \infty$.

We now show that $\Phi$ is a well-defined map. Let $\bY \in  \Sc_{\tilde{\F}^0}^2(I;\R^d)$. As $\bY \in \Sc^2_{\tilde{\F}^0}(I;\R^d)$ and from the assumptions on the model coefficients, we see that for fixed $(u,\omega) \in I \times \Omega$, the  mapping
$I \ni v \mapsto G_s(u,v)(\omega)Y_s^v(\omega),$
is measurable for every $s \in [t,T]$ and for $G \in \lbrace G^A,G^C \rbrace$. 
Now, from Picard iterations and assumption on model coefficients, we have 
\begin{align}\label{eq : picard to 0}
     \underset{u \in I}{\text{ess sup }} \E \Big[ \underset{s \in [t,T]}{\sup} \big| X_s^{\bY,u} - X_s^{\bY,u,n} \big|^2 \Big] \underset{ n \to \infty}{\to} 0,
\end{align}
where we define $X_s^{\bY,u,0} = \xi^u$ for all $s \in [t,T]$ and 
\begin{align}\label{eq : dynamics picard iterations}
    & X^{\bY,u,n}_s = \xi^u + \int_{t}^{s} \Big[\beta_s^u + A_s^u X_r^{\bY,u,n-1} +  \int_I G_s^A(u,v) Y_r^v  \d v + B_s^u \alpha^u_r \Big]  \d r \\
    &\qquad\qquad + \int_{t}^s\Big [\gamma_s^u + C_s^u  X^{\bY,u,n-1}_r +    \int_U G_{s}^C(u,v)  Y_r^v\d v + D_s^u \alpha^u_r \Big ] \d W^u_r+ \int_{t}^s \theta_s^u  \d \tilde{B}_r^0,\quad s\in[t,T]
\end{align}
and $\bar{X}_s^{\bY,u,n} = \E[ X_s^{\bY,u,n} | \tilde{B}^0_{. \wedge s}]$ for every $n \in \N^*$. Now from conditional Jensen's inequality we have
\begin{align}\label{eq : Picard iteration Xbar}
    \underset{u \in I}{\text{ess sup }} \E \Big[ \underset{s \in [t,T]}{\sup} \big| \bar{X}_s^{\bY,u} - \bar{X}_s^{\bY,u,n}     \big|^2 \Big]
    \leq \underset{u \in I}{\text{ess sup }} \E \Big[ \underset{s \in [t,T]}{\sup} | X_s^{\bY,u} - X_s^{\bY,u,n} |^2 \Big] \underset{n \to \infty}{\to} 0.
\end{align}

 We now show by induction that there exists Borel measurable functions $\phi^n : I \times [t,T] \times \Cc_{[0,T]} \to \R^d$ such that
 $$\bar{X}_s^{\bY,u,n} = \phi^{n}(u,s,\tilde{B}^0_{. \wedge s}), \quad \d u \otimes \d \P \text{ a.e.}, \forall s \in [t,T].$$
Let $n =0$. By construction, we have $X_s^{\bY,u,0} = \xi^u$ for every $s \in [t,T]$. From the definition of the conditional expectation, we can write
$    \E[\xi^u | \tilde{B}^0_{. \wedge s} = y] = \int_{\R^d} x \d \P_{\xi^u | \tilde{B}^0_{. \wedge s} = y}(\d x).$
However, the function $(u,s,y) \mapsto \P_{\xi^u | \tilde{B}^0_{. \wedge s}= y} \in \Pc(\R^d)$ is jointly measurable  thanks to  \eqref{eq : hypothesis initial conditions}. Indeed, defining the Borel map $\tilde{\xi}$ as  
$\tilde{\xi}(u,t,\omega,\omega^0,z) = (\xi(u,t,\omega_{. \wedge t},\omega^0_{. \wedge t},z), \omega^0),$
where $\xi$ is the Borel map defined in  \eqref{eq : hypothesis initial conditions},
 then denoting by $m$ the law of $U$, we have  $\P_{(\xi^u,\tilde{B}^0)} = \tilde{\xi}(u,t,\cdot,\cdot,\cdot)_{\sharp} \big(\W_T \otimes \W_T \otimes m\big)$  which is clearly measurable by assumption on $\xi$. Now, since the mapping $\pi : [0,T] \times \R^d \times \Cc_{[0,T]} \ni (s,x,b) \mapsto (x,b_{. \wedge s}) $ is continuous hence measurable, the map $(u,s) \mapsto \P_{(\xi^u, \tilde{B}^0_{. \wedge s})} = \pi(s,\cdot,\cdot)_{\sharp} \P_{(\xi^u,\tilde{B}^0)}$ is measurable  which gives the result by Remark \ref{rmk : well_posedness}.

We now show a suitable representation for $n \in \N^*$. Suppose it holds true for $n-1$. From the definition \eqref{eq : dynamics picard iterations}, it follows
\begin{align}\label{eq : relation for barX}
    \bar{X}_s^{\bY,u,n} &= \bar{\xi^u} + \E \Big[ \int_{t}^s \big[ \beta^u_r + A^u_r X_r^{\bY,u,n-1} +  \int_I G_{r}^A(u,v) Y_r^v  \d v + B^u_r \alpha^u_r \big]  \d r  \big | \tilde{B}_{. \wedge s}^0 \Big  ] +  \E \Big[ \int_{t}^s \theta_r^u  \d \tilde{B}^0_r  \big | \tilde{B}^0_{. \wedge s}\Big  ]
\end{align}
In the spirit of the proof of  \cite[Theorem 2.1]{kharroubi2025stochastic}, we notice  by a simple induction argument and from \cite[Lemma 10.1]{rogers2000diffusions} that there exists a measurable function $\psi^{\bY,n} : I \times [0,T]   \times \Cc_{[0,T]} \times \Cc_{[0,T]} \times \mathcal{K}$ such that for a.e. $u \in I$, we have
$X^{\bY,u,n}_s= \psi^{\bY,n}(u,s, W^u_{. \wedge s}, \tilde{B}^0_{. \wedge s},R),$ $  \P-\text{ a.s },$
for every $s \in [t,T]$. Combined with \eqref{eq : picard to 0}, we get a Borel map $\psi^{\bY} : I \times [0,T] \times \Cc_{[0,T]} \times \Cc_{[0,T]} \times\mathcal{K}$ by setting $\psi^{\bY} = \underset{n \to \infty}{\lim} \psi^{\bY,n}$ such that for a.e. $u \in I$, we have $X_s^{\bY,u} = \psi^{\bY}(u,s,W^u_{. \wedge s}, \tilde{B}^0_{. \wedge s},R), $ $ \P-\text{a.s.}$ which also gives the measurability of the map $u \to \P_{(X^{\bY,u},W^u, \tilde{B}^0,R)}$.

\noindent Moreover, since $\balpha \in \Ac$,  $\bY \in \Sc^2_{\tilde{\F}^0}(I;\R^d)$, and from the assumptions on the model coefficients, we notice that
\begin{align}
    \int_{t}^{s} \big[ \beta^u_r + A^u_r X_r^{\bY,u,n-1} +  \int_I G_{  A}(u,v) Y_r^v  \d v + B^u \alpha^u_r \big]  \d r = \tilde{\phi}^{1,n}(u,s, W^u_{. \wedge s}, \tilde{B}^0_{. \wedge s},R),
\end{align}
for a Borel measurable function $\tilde{\phi}^{1,n}$ and similarly for the stochastic integral term  which we  denote by $\tilde{\phi}^{2,n}$. Now,  again from the definition of the conditional expectation, we have  
\begin{align}
    \E \Big[\tilde{\phi}^{1,n}(u,s, W^u_{. \wedge s}, \tilde{B}^0_{. \wedge s},R) \big | \tilde{B}_{. \wedge s}^0 = y \Big] = \int_{\Cc_{[0,T]} \times \Cc_{[0,T]} \times \Kc} \tilde{\phi}^{1,n}(u,s,\omega_{. \wedge s},\omega^0_{. \wedge s},z) \P_{(W^u_{[0,T]}, \tilde{B}^0_{[0,T]},R)| \tilde{B}^0_{. \wedge s}= y}(\d \omega, \d \omega^0,\d z )
\end{align}
Now, again, we state that there exists a measurable version of  the map $(u,s,y) \mapsto \P_{(W^u, \tilde{B}^0,R)| \tilde{B}^0_{. \wedge s} = y}$. Indeed, the projection mapping $\pi :[t,T] \times \Cc_{[0,T]} \times \Cc_{[0,T]} \times\mathcal{K}  \ni (s,w,b,z) \mapsto (w,b,b_{. \wedge s},z)$ is continuous, hence measurable and therefore  the map $(u,s) \mapsto \P_{(W^u, \tilde{B}^0, \tilde{B}^0_{. \wedge s},R)}= \pi(s,\cdot,\cdot,\cdot)_{\sharp} \P_{(W^u,\tilde{B}^0,R)}$ is measurable. Therefore, since $\tilde{\phi}^{1,n}$ is Borel measurable and by Remark \ref{rmk : well_posedness}, we therefore have
a measurable version of the map
$(u,s,y) \mapsto \phi^{1,n}(u,s,y) = \E \Big[\tilde{\phi}^{1,n}(u,s, W^u_{. \wedge s}, \tilde{B}^0_{. \wedge s},R) \big | \tilde{B}_{. \wedge s}^0 = y \Big]$.
Since similar arguments holds for $\tilde{\phi}^{2,n}$, by defining $\phi^n = \bar{\xi}^u + \phi^{1,n} + \phi^{2,n}$, we therefore have
\begin{align}
    \bar{X}_s^{\bY,u,n} = \phi^{n}(u,s,\tilde{B}^0_{. \wedge s}), \quad \d u \otimes \d \P \text{ a.e.},
\end{align}
for every $s \in [t,T]$.
Now, equation \eqref{eq : Picard iteration Xbar} is enough to conclude for the representation of $\bar{X}^{\bY}$. Therefore, it shows that $\Phi$ is well-defined.

\noindent \textbf{Step 2  : } We now prove that the mapping $\Phi$ has a unique fixed-point. Let $\bY,\bZ \in \Sc^2_{\tilde{\F}^0}(I;\R^d)$. From standard computations, we have for a positive constant $C_T$ independent of $u \in I$ that 
\begin{align}
    \underset{s \in [t,r]}{\sup}\big |\bar{X}_s^{\bY,u} - \bar{X}_q^{\bZ,u} \big |^2 &\leq  C_T\int_{t}^{r} \Big[ \underset{t \leq q \leq l}{\sup}\big|A_q^u (\bar{X}_q^{\bY,u} - \bar{X}_q^{\bZ,u}) \big|^2 + \underset{t \leq q \leq l}{\sup}\big | \int_I G_q^{A}(u,v) (Y_q^v - Z_q^v) \d v \big|^2  \Big]  \d l. 
    &\quad 
\end{align}
Moreover, notice that 
\begin{align}
    \E &\Big[ \int_{t}^{r} \underset{t \leq q \leq l}{\sup}\int_I \Big | \int_I G_q^{A}(u,v) (Y_q^v - Z_q^v) \d v \Big |^2 \d u \d l \Big] = \E \Big[ \int_{t}^{r} \underset{t \leq q \leq l}{\sup} \int_I|T_{G_l^A}(Y_l-Z_l)(u)|^2 \d u \d l \Big] \\&= \E \Big[ \int_{t}^{r}\underset{t \leq q \leq l}{\sup}| T_{G_r^A}(Y_r-Z_r) |^2_{L^2(I;\R^d)} \d l \Big] \leq \E \Big[\underset{t \leq s \leq r}{\sup}\lVert T_{G_s^A} \rVert_{\Lc(L^2(I;\R^d))}^2  \int_{t}^{r}\int_I \underset{t \leq l \leq s}{\sup}|Y_l^u - Z_l^u|^2 \d u \d s  \Big], \\
    &\leq C_T \int_{t}^{r} \E \Big[\int_I   \underset{ t \leq l \leq  s}{\sup} | Y_l^u - Z_l^u |^2 \d u \Big] \d s,
\end{align}
where we used in the last inequality the assumption on $\rVert T_{G^A_{\cdot}} \rVert_{\Lc(L^2(I;\R^d))}$. Now, integrating with respect to $I$, taking the expectation and using the assumption on model coefficients,
\begin{align}
     \E \Big[ \int_I\underset{s \in [t,r]}{\sup}\big |\bar{X}_s^{\bY,u} - \bar{X}_q^{\bZ,u} \big |^2 \d u \Big] \leq C_T \int_{t}^{r} \E \Big[ \int_I \underset{t \leq q \leq l }{\sup} |\bar{X}_q^{\bY,u} - \bar{X}_q^{\bZ,u} |^2 \d u  + \int_I   \underset{ t \leq q \leq  l}{\sup} | Y_q^u - Z_q^u |^2 \d u \Big] \d l 
\end{align}
Applying Grönwall's Lemma to the function $r \to \E \Big[ \int_I\underset{s \in [t,r]}{\sup}\big |\bar{X}_s^{\bY,u} - \bar{X}_q^{\bZ,u} \big |^2 \d u \Big] $ gives
\begin{align}\label{eq : inequality_before_gronwall}
    \E \Big[ \int_I\underset{s \in [t,r]}{\sup}\big |\bar{X}_s^{\bY,u} - \bar{X}_q^{\bZ,u} \big |^2 \d u \Big] \leq C_T \int_{t}^{r} \E \Big[ \int_I   \underset{ t \leq q \leq  l}{\sup} | Y_q^u - Z_q^u |^2 \d u \Big] \d l ,
\end{align}
where $C_T$ only depends on model coefficients, namely $ \underset{\omega \in  \Omega}{\text{ ess sup}} \big(\underset{s \in [t,T]}{\sup} \lVert T_{G_s^A}(\omega) \rVert_{\Lc(L^2(I;\R^d))}\big)$ and $\underset{(u,\omega) \in I \times \Omega}{\text{ ess sup }} \big( \underset{  s\in [t,T]}{\sup} |A^u_s(\omega )| \big)$.

Setting $r=T$ in \eqref{eq : inequality_before_gronwall} leads to
\begin{align}
    \lVert \Phi(\bY) - \Phi(\bZ)  \rVert_{\Sc^2_{\tilde{\F}^0}(I;\R^d)}^2 = \int_I \E \Big[ \underset{s \in [t,T]}{\sup} \big| \bar{X}_s^{\bY,u} - \bar{X}_q^{\bZ,u} \big | \Big]^2 \d u &\leq C_T (T-t)   \int_I \E \Big[\underset{s \in [t,T]}{\sup} |Y_s^u - Z_s^u |^2  \d u \Big] \d s, \\
    &= C_T(T-t)  \lVert \bY-\bZ \rVert^2_{\Sc^2_{\tilde{\F}^0}(I;\R^d)} 
\end{align}
we conclude, by standard arguments that $\Psi$ has a unique fixed-point on the Banach space $\Sc^2_{\tilde{\F}^0}(I;\R^d)$.

 Now, denoting by $\bX = (X^u)_{u \in I}$ the process corresponding to the unique fixed-point $\bar{\bX}$, then it is clearly a solution to \eqref{eq : state equations LQ common noise} which belongs to $\Sc_t$ (recalling \eqref{eq : integrability_condition_X_appendix} for the integrability condition and up to a modification for the continuity). Moreover, if $\tilde{\bX} \in \Sc_t$ is another solution, then $(\bar{\tilde{X}}^u)_{u \in I}$ is a fixed-point of $\Psi$ and therefore coincides with $\bar{\bX}$ and therefore $X^u = \tilde{X}^u$ for $a.e.$ $u \in I$ since they are both solutions to \eqref{eq : state equations LQ common noise}. 

Now, the representation in \eqref{eq : representation_X} follows from similar arguments as used in Step 2 in the proof of Theorem 2.1 in \cite{kharroubi2025stochastic}. Finally, the estimate \eqref{eq :estimate_X} follows from standard estimates on SDEs.
\section{Proof of Proposition \ref{prop : fundamental_relation_common_noise}}\label{appendix : fundamental_relation}

\begin{proof}
 For $a.e.$ $u \in I$, we define the  random field  $(\Vc_s(u))_{t \leq s \leq T}$ as
    \begin{align}\label{eq : def process}
    \Vc_s(u) := \langle X_s^u ,K_s^u X_s^u \rangle + \langle \bar{X}_s^u ,\int_I  \bar{K}_s(u,v) \bar{X}_s^v \d v  \rangle + 2\langle Y_s^u, X_s^u \rangle + \Lambda_s^u 
\end{align}
We note the following identities which will be useful in the remainder in the proofs.
\begin{small}
\begin{align}\label{eq : remark_computation}
    \begin{cases}
        \int_I  \E \Big[ \Vc_T(u)| \tilde{\Fc}_t^0 \Big] \d u   = \int_I \E \Big[ \langle X_T^u ,H^u X_T^u \rangle  + \langle  \bar{X}_T^u ,\int_I G^H(u,v) \bar{X}_T^v \d v \rangle  \big | \tilde{\Fc}_t^0 \Big] \d u, \\
        \int_I \E \Big[ \Vc_t(u) | \tilde{\Fc}_t^0 \Big] \d u =  \int_I \bigg(\E \Big[ \langle \xi^u, K_t^u \xi^u \rangle |\tilde{\Fc}_t^0 \Big] +  \langle \bar{\xi}^u , \int_I \bar{K}_t(u,v) \bar{\xi}^v \d v \rangle + 2 \langle Y_t^u , \bar{\xi}^u \rangle + \Lambda_t^u \bigg) \d u.
    \end{cases}
\end{align}
\end{small}
\noindent We introduce the process $\bigg(  S_s^{\alpha}(u) := \Vc_s(u)  + \int_{0}^s  \Big[   \langle X_l^u, Q_l^u X_l^u \rangle + \langle \bar{X}_l^u, \int_I G^Q_l(u,v) \bar{X}_l^v \d v \rangle +\langle \alpha_l^u + I_l^u, R_l^u( \alpha_l^u + I_l^u) \rangle \Big] \d l  \bigg)_{0 \leq s \leq T} $ and the following notations 
\begin{align}
  \begin{cases}
    b_t^u = \beta_t^u +A_t^u X_t^u + \int_I G_t^A(u,v) \bar{X}_t^v \d v + B^u \alpha_t^u, \quad
     \sigma_t^u :=  \gamma_t^u + C_t^u X_t^u + \int_I G_t^C(u,v) \bar{X}_t^v \d v + D^u \alpha_t^u, \\
    \sigma_t^{0,u} :=\theta_t^u, \quad
    \bar{b}_t^{0,u} := \beta_t^u + A_t^u \bar{X}_t^u + \int_I G_t^A(u,v) \bar{X}_t^v \d v + B^u \bar{\alpha}_t^u, \quad
     \bar{\sigma}_t^{0,u} :=  \theta_t^u.
  \end{cases}  
\end{align}
For a fixed $u \in I$, we apply Itô's formula to the process $\big(S_s^{\alpha}(u)\big)_{0 \leq s \leq T}$, we integrate in time over $[t,T]$ and we take the conditional expectation with respect to $\Fc_t^{u,0}$ such that the conditional expectations of the  stochastic integral terms vanish. Then conditioning again by the lower $\sigma$-algebra $\tilde{\Fc}_t^0$ gives
\begin{align}\label{eq : Ito_formula_S_process}
    \E \big[\Sc_T^{\alpha}(u) |\tilde{\Fc}_t^0 \big] - \E \big[S_t^{\alpha}(u) | \tilde{\Fc}_t^0 \big] = \E \big[ \int_{t}^{T} D_s^{u,\alpha} \d s |\tilde{\Fc}_t^0 \big],
\end{align}
where we define 
\begin{small}
\begin{align}\label{eq : drif_term}
    D_s^{u,\alpha} &:= \langle b_s^u , K_s^u X_s^u \rangle  + \langle X_s^u, K_s^u b_s^u \rangle + \langle X_s^u, \dot{K}_s^u X_s^u \rangle + \langle X_s^u,Z_s^{K^u} \sigma_s^{0,u} \rangle + \langle \sigma_s^u, K_s^u \sigma_s^u \rangle \  + \langle \sigma_s^{0,u}, K_s^u \sigma_s^{0,u} \rangle + \langle \sigma_s^{0,u}, Z_s^{K^u} X_s^u \rangle  \notag \\
    &\quad + \int_I \Big[ \langle \bar{b}_s^u, \bar{K}_s(u,v) \bar{X}_s^v \rangle + \langle \bar{X}_s^u, \dot{\bar{K}}_s(u,v) \bar{X}_s^v \rangle +\langle \bar{X}_s^u, \bar{K}_s(u,v) \bar{b}_s^v \rangle + \langle \bar{X}_s^u, Z_s^{\bar{K}}(u,v) \bar{\sigma}_s^{0,v} \rangle \notag \\
    &\quad + \qquad  \langle \bar{\sigma}_s^{0,u}, \bar{K}_s(u,v) \bar{\sigma}_s^{0,v} \rangle+ \langle \bar{\sigma}_s^{0,u}, Z_s^{\bar{K}}(u,v) \bar{X}_s^v \rangle \Big] \d v \notag \\
    &\quad + \langle \dot{Y}_s^u, X_s^u \rangle + \langle Y_s^u , b_s^u \rangle + \langle Z_s^{Y}(u), \sigma_s^{0,u} \rangle  + \dot{\Lambda}_t^u  + \langle X_s^u, Q_s^u X_s^u \rangle + \langle \bar{X}_s^u, \int_I G_s^Q(u,v) \bar{X}_s^v \rangle + \langle \alpha_s^u + J_s^u, R_s^u (\alpha_s^u + J_s^u) \rangle \notag \\
    &=\langle X_s^u, \big(\dot{K}_s^u + \Phi_s^u\big) X_s^u \rangle   + \int_I  \Big[ \langle \bar{X}_s^u, \big( \dot{\bar{K}}_s(u,v) + \Psi_s(u,v) \Big) \bar{X}_s^v \Big] \d v + 2 \langle X_s^u, \dot{Y}_s(u) + M_s(u) \rangle + \dot{\Lambda}_s^u  + L_s^u \notag \\
    &\quad +  2  \langle U_s^u X_s^u + \int_I V_s(u,v) \bar{X}_s^v \d v  + \Gamma_s^u   , \alpha_s^u \rangle + \langle \alpha_s^u, O_s^u \alpha_s^u \rangle,
\end{align}
\end{small}
where we recall the notations introduced in
\eqref{eq : conditional_expectation_Xu},
\eqref{eq : notations_common_noise_Kt_no_control_vol}, \eqref{eq : notations_barKt_common_noise_no_control_vol}, \eqref{eq : notations_common_noise_Yt_no_control_vol} and \eqref{eq : coeff_Lambda} and where we introduced  $\dot{K}_s^u, \dot{\bar{K}}_s(u,v), \dot{Y}_s(u)$ and $\dot{\Lambda}_s^u$ as the $\d t $ part respectively in \eqref{eq : standardRiccati K}, \eqref{eq : abstract Riccati barK},\eqref{eq : linear_Riccati_Y_common_noise} and \eqref{eq : RiccatiLambda}. Notice that if the coefficients are deterministic as in Remark \ref{rmk : system_backward_equations} (4), then there is no need to introduce BSDE.

Denoting $\chi_t(\alpha_t^u) = 2  \langle U_t^u X_t^u + \int_I V_t(u,v) \bar{X}_t^v \d v  + \Gamma_t^u   , \alpha_t^u \rangle + \langle \alpha_t^u, O_t^u \alpha_t^u \rangle$,  after doing a square completion, we have
\begin{small}
\begin{align}\label{eq :  square_completion_chi_common_noise}
     \chi_t(\alpha_t^u) &=  \Big \langle O_t^u \Big(\alpha_t^u +  (O_t^u)^{-1} \big( U_t^u X_t^u + \int_I V_t(u,v) \bar{X}_t^v \d v  + \Gamma_t^u )\Big) , \alpha_t^u +  (O_t^u)^{-1} \big( U_t^u X_t^u + \int_I V_t(u,v) \bar{X}_t^v \d v  + \Gamma_t^u  \big) \Big \rangle  \notag \\
    &\qquad - \big \langle U_t^u X_t^u , (O_t^u)^{-1} U_t^u X_t^u \big \rangle - \int_I \big \langle U_t^u X_t^u, (O_t^u)^{-1} V_t(u,v) \bar{X}_t^v \big \rangle \d v - \langle U_t^u X_t^u, (O_t^u)^{-1} \Gamma_t^u \rangle \rangle  \notag \\
    &\qquad - \int_I  \big \langle V_t(u,v) \bar{X}_t^v,(O_t^u)^{-1} U_t^u \bar{X}_t^u \big \rangle \d v  - \int_I \langle \bar{X}_t^u,   \int_I V_t(w,u)^{\top} (O_t^w)^{-1} V_t(w,v) \bar{X}_t^v \d w \rangle \d v - \int_I \langle \bar{X}_t^u, V_t(v,u)^{\top} (O_t^v)^{-1} \Gamma_t^v \rangle \d v \notag \\
    & \qquad - \langle (U_t^u)^{\top} (O_t^u)^{-1} \Gamma_t^u, X_t^u \rangle - \langle \int_I  V_t(v,u)^{\top} (O_t^v)^{-1}\Gamma_t^v \d v , X_t^u \rangle  -  \langle \Gamma_t^u, (O_t^u)^{-1} \Gamma_t^u \rangle 
\end{align}
\end{small}
\noindent Plugging \eqref{eq :  square_completion_chi_common_noise}  into \eqref{eq : drif_term} gives
\begin{small}
\begin{align}\label{eq : drift_term_square_completion}
    D_s^{u,\alpha} &= \langle X_s^u , (\dot{K}_s^u + \Phi_s(u) - (U_s^u)^{\top} (O_s^u)^{-1} U_s^u) X_s^u \rangle \notag \\
    &\quad +\int_I \Big[ \langle\bar{X}_s^u, \big( \dot{\bar{K}}_s(u,v) + \Psi_s(u,v) - (U_s^u)^{\top} (O_s^u)^{-1} V_s(u,v)  - V_s(v,u)^{\top} (O_s^v)^{-1}  U_s^v \big) \bar{X}_s^v \rangle \Big] \d v  \notag \\
    &\quad +  2 \langle X_s^u, \dot{Y}_s(u) + M_s(u) - (U_s^u)^{\top} (O_s^u)^{-1} \Gamma_s^u - \int_I V_s(v,u)^{\top} (O_s^v)^{-1} \Gamma_s^v \d v  \rangle \notag \\
    &\quad + \dot{\Lambda_s^u} + L_s^u - \langle \Gamma_s^u, (O_s^u)^{-1} \Gamma_s^u \rangle  \notag \\
    &\quad + \Big \langle O_s^u \Big(\alpha_s^u +  (O_s^u)^{-1} \big( U_s^u X_s^u + \int_I V_s(u,v) \bar{X}_s^v \d v  + \Gamma_s^u )\Big) , \alpha_s^u +  (O_s^u)^{-1} \big( U_s^u X_s^u + \int_I V_s(u,v) \bar{X}_s^v \d v  + \Gamma_s^u  \big) \Big \rangle.
\end{align}
\end{small}
Using now \eqref{eq : remark_computation}, \eqref{eq : drift_term_square_completion} and  since $(K,\bar{K},Y,\Lambda)$ are respectively solutions to \eqref{eq : standardRiccati K}, \eqref{eq : abstract Riccati barK}, \eqref{eq : linear_Riccati_Y_common_noise} and \eqref{eq : RiccatiLambda}, we get recalling \eqref{eq : Ito_formula_S_process}  and integrating over $I$
\begin{small}
\begin{align}
     &\int_t^T  \int_I\E\Big[\Big\langle O^u_s\Big(\alpha^u_s + (O^u_s)^{-1}\big(U^u_sX^u_s + \int_I V_s(u,v)\bar X^v_s \d v + \Gamma_s^u \big)\Big),    \alpha^u_s + (O^u_s)^{-1}\big(U^u_sX^u_s + \int_I V_s(u,v)\bar X^v_s \d v + \Gamma_s^u  \big)\Big\rangle \big | \tilde{\Fc}_t^0 \Big]\d s  \\
    &=\int_I \E \Big[ \Vc_T(u) | \tilde{\Fc}_t^0 \Big] \d u - \int_I \E \Big[ \Vc_t(u) | \tilde{\Fc}_t^0 \Big] \d u + \int_I \E \Big[\int_{t}^{T} \Big[   \langle X_s^u, Q_s^u X_s^u \rangle + \langle \bar{X}_s^u, \int_I G_s^Q(u,v) \bar{X}_s^v \d v \rangle +\langle \alpha_s^u, R_s^u \alpha_s^u \rangle  \d s \big | \tilde{\Fc}_t^0 \Big] \d u  \\
    &= J(t,\boldsymbol{\xi},\balpha) - \Vc(t,\boldsymbol{\xi}),
\end{align}
\end{small}
where we used that $\int_I \E \Big[ \Vc_t(u) | \tilde{\Fc}_t^0 \Big] \d u = \Vc(t,\boldsymbol{\xi})$ and where all the equalities hold $\P \text{-a.s}$.  The proof is therefore completed.
\end{proof}
\begin{small}
\paragraph*{Acknowledgements.}
The authors are grateful to Huyên Pham and the two anonymous
referees for their careful reading of the paper and useful remarks.
\end{small}

\renewcommand{\thesection}{\Alph{section}}

\begin{small}
\bibliographystyle{plain}   
\bibliography{ReferencesTest.bib}   
\end{small}

\end{document}